\documentclass{amsart}
\usepackage{amsthm,amsmath,amssymb,amscd,graphics,enumerate,stmaryrd}
\usepackage[all]{xypic}

{
   \newtheorem{theorem}[subsubsection]{Theorem}
   \newtheorem{proposition}[subsubsection]{Proposition}     
   \newtheorem{lemma}[subsubsection]{Lemma}
   
   \newtheorem{corollary}[subsubsection]{Corollary}

}
{\theoremstyle{definition}
   
   \newtheorem{example}[subsubsection]{Example}
   \newtheorem{definition}[subsubsection]{Definition}
   \newtheorem{remark}[subsubsection]{Remark}
   
   \newtheorem{assumption}[subsubsection]{Assumption}
}

\newcommand{\QQ}{{\mathbb{Q}}}
\newcommand{\NN}{{\mathbb{N}}}
\newcommand{\PP}{{\mathbb{P}}}
\newcommand{\ZZ}{{\mathbb{Z}}}

\newcommand{\GG}{{\mathbb{G}}}

\newcommand{\bbA}{{\mathbb{A}}}

\newcommand{\bR}{{\mathbf{R}}}

\newcommand{\bmu}{{\boldsymbol{\mu}}}

\newcommand{\fF}{{\mathfrak{F}}}

\newcommand{\cB}{{\mathcal B}}

\newcommand{\cE}{{\mathcal E}}
\newcommand{\cF}{{\mathcal F}}
\newcommand{\cG}{{\mathcal G}}
\renewcommand{\cH}{{\mathcal H}}
\newcommand{\cI}{{\mathcal I}}
\newcommand{\cK}{{\mathcal K}}
\renewcommand{\cL}{{\mathcal L}}

\newcommand{\cO}{{\mathcal O}}
\newcommand{\cP}{{\mathcal P}}

\newcommand{\WP}{{\boldsymbol\cP}}
\newcommand{\cQ}{{\mathcal Q}}
\renewcommand{\cR}{{\mathcal R}}
\newcommand{\cS}{{\mathcal S}}

\newcommand{\cW}{{\mathcal W}}
\newcommand{\cX}{{\mathcal X}}
\newcommand{\cY}{{\mathcal Y}}

\newcommand{\slc}{{\operatorname{slc}}}
\newcommand{\can}{{\operatorname{can}}}
\newcommand{\cm}{{\operatorname{cm}}}
\newcommand{\gor}{{\operatorname{gor}}}
\newcommand{\gorslc}{{\text{gor-slc}}}

\newcommand{\StaL}{{\cS}ta^\cL}
\newcommand{\OrbL}{{{\cO}rb^\cL}} 
\newcommand{\Orbl}{{\cO}rb^\lambda}

\newcommand{\Orbo}{{\cO}rb^\omega}

\newcommand{\KolL}{\cK^\cL} 
\newcommand{\Koll}{\cK^\lambda} 
 
\newcommand{\Kolo}{\cK^\omega}

\newcommand{\Spec}{\operatorname{Spec}}
\newcommand{\Proj}{\operatorname{Proj}}
\newcommand{\cProj}{{{\WP}\boldsymbol{roj}}}
\newcommand{\Isom}{\operatorname{Isom}}
\newcommand{\Pic}{{\operatorname{\mathbf{Pic}}}}
\newcommand{\Hom}{{\operatorname{Hom}}}
\newcommand{\cHom}{{{\cH}om}}

\newcommand{\Aut}{{\operatorname{Aut}\,}}
\newcommand{\Sym}{{\operatorname{Sym}}}
\newcommand{\codim}{\operatorname{codim}}
\newcommand{\lrar}{\longrightarrow}
\newcommand{\dar}{\downarrow}

\newcommand{\PGL}{{\operatorname{\mathbf{PGL}}}}

\newcommand{\cHilb}{{{\cH}ilb}}

\newcommand{\lcm}{{\operatorname{lcm}}}

\newcommand{\setmin}{\smallsetminus}
\newcommand{\thickslash}{\mathbin{\!\!\pmb{\fatslash}}}

\renewcommand{\setminus}{\setmin}

\begin{document}

\title[Stable varieties
]{Stable varieties with a twist}

\author[D. Abramovich]{Dan Abramovich}
 \thanks{Research of D.A. partially supported by NSF grant DMS-0335501 and DMS-0603284}
  
\address{Department of Mathematics, Box 1917, Brown University,
Providence, RI, 02912, U.S.A} 
\email{abrmovic@math.brown.edu}

\author[B. Hassett]{Brendan Hassett}
\thanks{Research of B.H. partially supported by NSF grants DMS-0196187, DMS-0134259, and DMS-0554491 and by the Sloan Foundation}  
\address{Department of Mathematics, MS 136,
Rice University,
Houston, TX 77251, U.S.A}
\email{hassett@math.rice.edu}
\date{April 17, 2009}

\maketitle

\setcounter{tocdepth}{1}
\tableofcontents

\section{Introduction}

\subsection{Moduli of stable varieties: the case of surfaces}
In the paper \cite{K-SB}, Koll\'ar and Shepherd-Barron introduced {\em stable surfaces } as a generalization
of stable curves. 
This class is natural from the point of view of the minimal model program, which shows that
any one-parameter family of surfaces of general type admits a unique stable limit.  Indeed, the stable
reduction process of Deligne and Mumford can be interpreted using the language of minimal models of surfaces.
Stable surfaces admit {\em semilog canonical singularities}.
   Codimension-one semilog canonical singularities
are nodes, but in codimension two more complicated singularities arise.  However, these singularities are all reduced, 
satisfy Serre's condition $S_2$ (hence are Cohen--Macaulay in codimension 2),
and admit a well-defined $\QQ$-Cartier canonical divisor $K_S$.
Stable varieties in any dimension are similarly defined
as proper varieties with semilog canonical singularities and ample canonical divisor.

Koll\'ar and Shepherd-Barron proposed the moduli space of stable surfaces as the natural compactification
of the moduli space of surfaces of general type.  Much of this was established in \cite{K-SB}, which also
offered a detailed classification of the singularities that arise.  Boundedness of the class of stable 
surfaces with fixed invariant $(K_S)^2$ was shown by Alexeev \cite{Alexeev}. 

\subsection{The problem of families}
To a great extent this established the moduli space of stable surfaces as the `right' 
geometric compactification of moduli of surfaces of general type.  But one nagging question remains
unresolved: {\em what flat families $\pi:X \to B$ of stable surfaces should be admitted
in the compactification?}

The problem stems precisely from the fact that the dualizing sheaf
$\omega_S = \cO_S(K_S)$ is not necessarily  invertible for a stable surface.
Rather, for a suitable positive integer $n$ {\em depending on $S$}, the sheaf 
 $$ \cO_S(nK_S) \ \  = \ \ \omega_S^{[n]} \ \ :=\ \  ({\omega_S^n})^{**} $$ 
is invertible.  As it turns out, there exist flat families $\pi:X \to B$ 
over a smooth curve where the fibers
are stable surfaces but $\omega_{\pi}^{[n]}$ is not locally free for any $n\neq 0$.  For a local example, see Pinkham's analysis of the deformation space of the cone over a 
rational normal quartic curve described in \cite[Example 2.8]{K-SB}.

Two approaches were suggested to resolve this issue (see \cite{HK} for detailed
discussion). 
In \cite[5.2]{KollarJDG}, Koll\'ar proposed what has come to be called a {\em Koll\'ar family} 
of stable surfaces, where the formation of {\em every} saturated power of $\omega_{\pi}$ is required to commute 
with base change.  
Viehweg \cite{Viehweg} suggested
using families where {\em some} saturated power of $\omega_{\pi}$ is invertible
(and in particular, commutes with base extension), i.e., 
there exists an integer $n>0$ and an invertible sheaf on $X$
restricting to $\omega_{S}^{[n]}$ on each fiber $S$ of $\pi$.  

It has been shown in \cite{HK} that Viehweg families give rise to a good moduli space.
The purpose of this article is to address the case of Koll\'ar families.  While this
question has been considered before (cf. \cite{Hacking}), the solution we propose is
perhaps more natural than previous approaches, and has implications beyond the question
at hand. 

We must note that Koll\'ar's \cite{KollarHH} resolves the issue by tackling it head on using the moduli spaces of Husks. Our approach goes by way of showing that, in the situation at hand, if we consider  objects with an appropriate stack structure, the problem does not arise. This does  require us to show that moduli spaces of such stack structures are well behaved. 

\subsection{Canonical models and quotient stacks}
Here is the basic idea:  Let $X$ be a smooth projective variety of general type.
A fundamental invariant of $X$ is the {\em canonical ring}
$$R(X):=\oplus_{n\ge 0} \Gamma(X,\omega_X^n),$$
a graded ring built up from the differential forms on $X$.  This ring has long
been known to be a birational invariant of $X$ and has recently been shown to 
be finitely generated \cite{BCHM}.  This was known classically for curves and
surfaces, and established for threefolds by S. Mori in the 1980's.
The associated projective variety
$$X^{\can}:=\Proj R(X)= (\Spec R(X) \setminus 0 ) / \GG_m$$
is called the {\em canonical model} of $X$.  Some saturated power
$\omega^{[n]}_{X^{\can}}$ is ample and invertible, essentially by construction.   
Therefore, deformations of canonical models (and stable surfaces)
are most naturally expressed via deformations of the canonical ring.

In these terms, our main problem is that certain deformations of a stable surface
do not yield deformations of its canonical ring.  Our solution is to replace $X^{\can}$
with a more sophisticated geometric object that remembers its algebraic origins.
Precisely, we consider quotient {\em stacks}
$$\cX^{\can}=\cProj R(X) := \left[ (\Spec R(X) \setminus 0 ) / \GG_m\right],$$
which admit nontrivial stack-structure at precisely the points where $\omega_{X^{\can}}$
fails to be invertible.  Our main goal is to show that this formulation is equivalent
to Koll\'ar's, and yields a workable moduli space.

\subsection{Moduli of stacks: the past}
The idea that families with fibers having stack structure should admit workable moduli spaces is not new. On the level of deformations, Hacking's approach in \cite{Hacking} uses precisely this idea. The case of {\em stable fibered surfaces}, namely stable surfaces fibered over curves with semistable fibers, was discussed in \cite{AV1}. It was generalized to many moduli problems involving curves in \cite{AV2} and to varieties of higher dimension with suitable ``plurifibration'' structures \cite{A-pluri}. The fibered and plurifibered examples are a little misleading, since the existence of canonical models and complete moduli in those cases are an outcome of \cite{AV2} and do not require the minimal model program. Still they do reinforce the point that stacks do not pose a true obstacle in defining moduli.

We show in this paper that this is indeed the case for stacks of the form  $\cX = \cProj R$ for $R$ a finitely generated graded ring, and moreover this works in exactly the same way it works for varieties, if one takes the appropriate viewpoint.
Classical theory of moduli of projective varieties starts by considering varieties embedded in $\PP^n$, forming Hilbert schemes, taking quotient stacks of the appropriate loci of normally embedded varieties by the action of $\PGL(n+1)$, and taking the union of all the resulting stacks over all types of embeddings. Then one tries to carve out pieces of this ``mother of all moduli spaces'' (of varieties) which are bounded, and if lucky, proper, by fixing appropriate numerical invariants.

Miles Reid has been advocating over the years that, 
while most people insist on embedding projective varieties in projective 
space, the varieties really beg to be embedded in {\em weighted} projective spaces, denoted $\PP(\rho)$ below. This is simply because  graded rings are in general not generated in one degree;
see \cite{Reid78}, \cite{Reid-Chapters}, and especially \cite{Reid-Graded}.  
For our stacks this point is absolutely essential, and in fact our stacks (as well as varieties) are embedded in the appropriate weighted projective {\em stacks}, denoted  $\WP(\rho)$ below.  

\subsection{Moduli of stacks: this paper}
In section \ref{Sec:weighted} we consider weighted projective stacks and their substacks. Following \cite{OS} substacks are parametrized by a Hilbert scheme. All substacks of $\WP(\rho)$ are cyclotomic, in particular tame, and in analogy with the scheme case, we characterize in Proposition \ref{Prop:criterion} pairs $(\cX, \cL)$ consisting of a cyclotomic stack with a {\em polarizing line bundle} - one affording an embedding - in cohomological terms. 

Similar results are obtained in \cite{Ross-Thomas}, though the emphasis is different.  Ross and Thomas embed smooth orbifolds in weighted projective stacks
to study orbifold constant scalar curvature
K\"ahler metrics.  They develop a notion of Geometric Invariant
Theory stability for polarized orbifolds and use the relation between
stability and moment maps to formulate obstructions to existence of these metrics.

In section \ref{Sec:StaL} we use these results to construct an algebraic stack $\StaL$ parametrizing pairs $(\cX,\cL)$ of a proper stack with polarizing line bundle $\cL$  (see Theorem \ref{Th:PP-algebraic}). This involves studying  the effect of changing the range of weights on the moduli of normally embedded substacks (see Proposition \ref{Prop:change-n-m}). 

In section \ref{Sec:Orb} we introduce {\em orbispaces}, discuss the difference between {\em polarizing line bundle} and {\em polarization}, and construct  algebraic stacks $\OrbL$ and $\Orbl$ parametrizing {\em orbispaces  with polarizing line bundles} and {\em polarized orbispaces}, respectively (see Proposition \ref{Prop:OrbL}).
We further discuss the canonical polarization, and the algebraic stack $\Orbo$ of canonically polarized orbispaces, in Theorem \ref{Th:OrbOm}. 

This ends the general discussion of moduli of stacks, and brings us to the questions related to Koll\'ar families. We begin Section \ref{Sec:Kollar} by discussing Koll\'ar families of $\QQ$-line bundles (see Definition \ref{Def:Kollar}). We then define {\em uniformized twisted varieties} in Definition \ref{Def:twisted-variety}. An important point here is that {\em the conditions on flat families are entirely on geometric fibers}. 

Finally we relate the two notions and prove an equivalence of categories between Koll\'ar families $(X\to B, F)$ of $\QQ$-line bundles  and families of uniformized twisted varieties (see Theorem \ref{Th:KQisPP}). In particular this means that we have an algebraic stack $\KolL$ of Koll\'ar families with a polarizing $\QQ$-line bundle, which is at the same time the stack of twisted varieties with polarizing line bundles. The rigidification $\Koll$ of $\KolL$ is naturally the stack  of polarized twisted varieties.
Passing to the locally closed substack where the initial homology group of the relative dualizing complex is invertible and polarizing, 
we obtain the substack $\Kolo$ of {\em canonically polarized twisted varieties}. 

We end the paper by applying the previous discussion to moduli of stable varieties in Section \ref{Sec:SLC}. We define the moduli functor in terms of families of twisted stable varieties and equivalently Koll\'ar families stable varieties. We show that this is open in $\Kolo$  (see Proposition \ref{prop:SLC})  and discuss what is known about properness in section \ref{Sec:properness}. 

It may be of interest to develop a treatment of the {\em steps} of the  minimal model program, and not only the end product $\cProj R(X)$,  using stacks. It would require at least an understanding of positivity properties for the birational ``contraction'' $\cX\to X$ from a twisted variety to its coarse moduli space, where we expect the notion of {\em age}, introduced in \cite{IR}, to be salient.  This however goes beyond the scope of the present paper.

\subsection{Acknowledgments}  We are grateful to 
Valery Alexeev, 
Kai Behrend,
Tim Cochran, 
Alessio Corti,
Barbara Fantechi, 
J\'anos Koll\'ar, 
S\'andor Kov\'acs,
Martin Olsson,
Miles Reid,
Julius Ross,
Richard Thomas, 
Eckart Viehweg,
and 
Angelo Vistoli
for helpful conversations about these questions.  We appreciate the hospitality of the
Mathematisches Forschungsinstitut Oberwolfach and
the Mathematical Sciences Research Institute in Berkeley, California.

\subsection{Conventions}
Through most of this paper we work over $\ZZ$.  When we refer to schemes of finite type, we mean schemes of
finite type over $\ZZ$.   The main exceptions are Section~\ref{Sec:SLC} and the Appendix,
where we assume characteristic zero.
We leave it to the interested reader to the supply the incantations needed to extend our results to
an arbitrary  base scheme.

In characteristic zero, we freely use the established literature for Deligne-Mumford stacks.  
In positive and mixed characteristics, we rely on the recent paper \cite{AOV}, which develops a
new notion of `tame stacks' with nice properties, e.g., they admit coarse moduli spaces that 
behave well under base extension.

\subsection{Notation} \verb. .\\

\begin{tabular}{lll}
$\PP(\rho)$&	
			Weighted projective space of $\GG_m$-representation $\rho$ \\
$\WP(\rho)$&	
			Weighted projective {\em stack} of $\GG_m$-representation $\rho$  \\
${\StaL}$& 	
			Stack of {\em stacks} with polarizing line bundle\\
${\OrbL}$&	
			Stack of  {\em orbispaces} with polarizing line bundle\\ 
${\Orbl}$&	
			-- with polarization\\
${\Orbo}$&	
			-- with canonical polarization\\
${\KolL}$&	
			Stack of {\em twisted varieties} with polarizing line bundle\\
${\Koll}$&	
			--  with polarization\\
${\Kolo}$&	
			--  with canonical polarization\\
$\Box_\cm$&	
			Index indicating  Cohen-Macaulay fibers\\ 
$\Box_\gor$&	
			--  Gorenstein fibers\\ 
$\Box_\can$&	
			--  canonical singularities (in the coarse fibers)\\ 
$\Box_\slc$&	
			--  semilog canonical singularities fibers\\ 
$\Box_\gorslc$&	
			--  Gorenstein s.l.c. singularities (in the twisted fibers) \\ 
$\Box^\can$&	
			indicating canonical model of something \\ 
\end{tabular}

\section{Cyclotomic stacks and weighted projective stacks}\label{Sec:weighted}

\subsection{Weighted projective stacks}

\begin{definition}
Fix a nondecreasing sequence of positive integer weights
$$0 <  \rho_1 \le \ldots \le \rho_r$$
and consider the associated linear action of $\GG_m$
on affine space
\begin{eqnarray*}
\rho: \GG_m \times \bbA^r & \rightarrow &\bbA^r \\
\rho^*(x_1,\ldots,x_r) & = &  
	(t^{\rho_1}x_1,\ldots,t^{\rho_r}x_r).
\end{eqnarray*}
Recall that {\em weighted projective space} associated to the representation $\rho$
is defined as the quotient scheme
$$\PP(\rho)\ \ = \ \ \PP(\rho_1,\ldots,\rho_r)\ \ \ =\ \ \ 
(\bbA^r\setminus 0)\,/\,\GG_m.$$
We define the {\em weighted projective stack} as the quotient stack
$$\WP(\rho)\ \ =\ \ \WP(\rho_1,\ldots,\rho_r)\ \ \ =\ \ \  
\big[\,(\bbA^r\setminus 0)\ /\ \GG_m\,\big].$$
\end{definition}

Each weighted projective stack has the corresponding
weighted projective space as its coarse moduli space.
For each positive integer $d$ we have
$$\PP(d\rho_1,\ldots,d\rho_r)\ \ \simeq\ \ 
\PP(\rho_1,\ldots,\rho_r),$$
however, the weighted projective stacks are not isomorphic unless $d=1$. Indeed $\PP(d\rho_1,\ldots,d\rho_r)\to  
\PP(\rho_1,\ldots,\rho_r)$ is a gerbe banded by $\bmu_d$.  Note also that a weighted projective stack is
representable if and only if each weight $\rho_i=1$;
it contains a nonempty representable open subset if and only if
$\gcd(\rho_1,\ldots,\rho_r)=1$.  

\begin{definition}
Let $B$ be a scheme of finite type and $V$ 
a locally-free $\cO_B$-module of rank $r$.  Suppose
that $\GG_m$ acts $\cO_B$-linearly on $V$ with positive weights
$(\rho_1,\ldots,\rho_r)$.  Let $w_1<\ldots<w_m$ denote
the distinct weights that occur, so we have a decomposition
of $\cO_B$-modules
$$V\ \ \simeq\ \  V_{w_1} \oplus \ldots \oplus V_{w_m}$$
with
$$\rho(v) = t^{w_i}v, \quad v\in V_{w_i}.$$
Write $\bbA  := \Spec_B ( \Sym_B^\bullet V)$ for  the associated vector bundle, with its zero section $0\subset \bbA$,  and
$$
\rho: \GG_m \times_B \bbA  \rightarrow \bbA
$$
for the associated group action.
The {\em weighted projective stack} associated with $V$ and $\rho$
is the quotient stack
$$\WP( V,\rho)\ \ :=\ \ 
\left[\,(\bbA\setminus 0)\ /\ \GG_m\,\right]\  \rightarrow \ B.$$
We will sometimes drop $V$ from the notation when  the meaning  is unambiguous.  
\end{definition}

Graded coherent $\cO_{\bbA}$-modules descend 
to coherent sheaves on $\WP(V,\rho)$;  thus
for each integer $w$, the twist $\cO_{\bbA}(w)$ yields an
{\em invertible} sheaf $\cO_{\WP(V,\rho)}(w)$ on $\WP(V,\rho)$.  
The canonical graded homomorphisms
$$V_{w_i}\otimes_B \cO_{\bbA} \rightarrow \cO_{\bbA}(w_i)$$
induce
$$\phi_{w_i}:V_{w_i} \otimes_B \cO_{\WP(V,\rho)} \rightarrow  \cO_{\WP(V,\rho)}(w_i),
\quad i=1,\ldots,m.$$
Note that $\phi_{w_i}$ vanishes along the locus  where
the elements of $V_{w_i}$ are simultaneously zero, i.e., where all
the weight-$w_i$ coordinates simultaneously vanish. The homomorphisms $\phi_{w_i}, i=1,\ldots,m$ do not vanish  simultaneously anywhere.

The following lemma is well-known in the case of projective space.  
The proof for weighted-projective stacks is identical:
\begin{lemma} The stack $\WP(V,\rho)$
is equivalent, as a category fibered over the category of 
$B$-schemes, with the following category:
\begin{enumerate} 
\item Objects over a scheme $T$ consist of 
   \begin{enumerate}
   \item a line bundle $L$ over $T$, and
   \item $\cO_T$-linear homomorphisms 
      $$\phi_{w_i}:V_{w_i}\otimes_B \cO_T \to L^{w_i},\quad i=1,\ldots,m,$$
	not vanishing simultaneously at any point of $T$.
\end{enumerate}
\item Arrows consist of fiber diagrams 
$$
\begin{array}{ccc}
L' & \rightarrow  & L\\
\downarrow & & \downarrow \\
T' & \rightarrow & T 
\end{array}
$$
compatible with the homomorphisms $\phi_{w_i}$ and $\phi'_{w_i}$.
\end{enumerate}
\end{lemma}

As we shall see, certain stacks admitting ``uniformizing line bundles" have
representable morphisms into weighted projective stacks.

\subsection{Closed substacks of weighted projective stacks}
By descent theory, closed substacks 
$$\cX \ \ \subset\ \  \WP(\rho)=\WP(\rho_1,\ldots,\rho_r)$$
correspond to closed subschemes
$$\cX\times_{\WP(\rho)}(\bbA^r\setminus 0)\ \  \subset\ \  
(\bbA^r\setminus 0),$$
equivariant under the action of $\GG_m$.  Let $C\cX$
denote the closure of this scheme in $\bbA^r$, i.e.,
the cone over $\cX$. 
Each such subscheme over a field $k$  can be defined by a graded ideal
$$I\subset k[x_1,\ldots,x_r].$$
The {\em saturation} of such an ideal is defined as 
$$\{f \in k[x_1,\ldots,x_r]: \left<x_1,\ldots,x_r\right>^df
\subset I \text{ for some } d\gg 0 \}$$
and two graded ideals are {\em equivalent} if they have
the same saturation.   Equivalent graded ideals give identical ideal sheaves on $\bbA^r\setminus 0$, hence they define the
same equivariant subschemes of $\bbA^r\setminus 0$ corresponding to the same substack of  $\WP(\rho)$.

There is one crucial distinction from the standard theory
of projective varieties:  when the weights  $\rho_i=1$  for all $i$, 
every graded ideal is equivalent to an ideal generated
by elements in a single degree.  This can fail when the weights
are not all equal.  However, if we set $N=\lcm(\rho_1,\ldots,\rho_r)$
then every graded ideal is equivalent to one
with generators in degrees $n,n+1,\ldots,n+N-1$, for $n$ sufficiently
large. (This can be be shown by regarding graded ideals as modules 
over the subring of $k[x_1,\ldots,x_r]$ generated by monomials of 
weights divisible by $N$.) An equivariant subscheme $C\cX\subset \bbA^r$ is specified
by the induced quotient homomorphisms on the graded pieces 
of the coordinate rings
$$k[x_1,\ldots,x_r]_d \twoheadrightarrow k[C\cX]_d, \quad
d=n,n+1,\ldots,n+N-1.$$

In sheaf-theoretic terms, $\cO_{\WP(\rho)}$-module quotients
$$\cO_{\WP(\rho)}  \twoheadrightarrow \cO_{\cX}$$
are determined by the induced quotient of $\cO_{\PP(\rho)}$-modules
$$\pi_*[\oplus_{i=0}^{N-1} \cO_{\WP(\rho)}(i)] \rightarrow
\pi_*[\oplus_{i=0}^{N-1} \cO_{\cX}(i)],$$ or any other interval of twists of length $N$, 
where $\pi:\WP(\rho) \rightarrow \PP(\rho)$ is the coarse moduli morphism.  

\begin{definition}\label{Def:normally}
Fix a representation $(V,\rho)$  of $\GG_m$ with isotypical components $V_{w_1},\ldots,V_{w_m}$ of positive weights $w_1<\cdots<w_m$. 

A closed  substack $\cX \subset \WP(\rho)$  is said to be {\em normally embedded} if
\begin{enumerate}
\item for each $1\leq j\leq m$, the homomorphism $V_{w_j} \to H^0(\cX,\cO_\cX(w_j))$ is an isomorphism,
\item for each $k\geq w_1$ and $i\geq 1$ we have $H^i(\cX,  \cO_\cX(k))=0$, and
\item  the natural homomorphism
$$\Sym\, V \ \ \lrar\ \  \oplus_{k\geq 0} H^0(\cX, \cO_{\cX}(k))$$
surjects onto the component $H^0(\cX,  \cO_\cX(k))$ for each $k\geq w_1$.
\end{enumerate}
\end{definition}

We proceed, following Olsson--Starr \cite{OS} to classify substacks $\cX \subset \WP(\rho)$ by a Hilbert scheme. The following definition is \cite[Definition 5.1]{OS}.

\begin{definition}
A locally free sheaf $\cE$ on a stack $\cY$
is {\em generating}  if, for each quasi-coherent $\cO_{\cY}$-module
$\cF$, the natural homomorphism
$$\theta(\cF):\pi^*\pi_*\cHom_{\cY}(\cE,\cF) \otimes_{\cY} \cE
\rightarrow \cF$$
is surjective.
\end{definition}

We note  that
$$\cE=\oplus_{i=0}^{N-1} \cO_{\WP(\rho)}(-i)$$
is a generating sheaf for $\WP(\rho)$ 
where $N = \lcm(w_1,\ldots,w_m)$.
The main insight of \cite{OS} \S 6
is to reduce the study of 
$\cO_{\WP(\rho)}$-quotients of $\cF$ to an analysis of
$\cO_{\PP(\rho)}$-quotients of $\pi_*\cHom_{\WP(\rho)}(\cE,\cF).$
The following is a direct application of \cite[Theorem 1.5]{OS} to weighted projective stacks.

\begin{theorem}
Let $\rho$ be an action of  $\GG_m$ on $\bbA^r$ with positive weights
and $\WP(\rho)$ the resulting weighted projective stack.
For each scheme $T$ and stack $\cY$, write
$\cY_T=\cY\times T$.
Consider the functor 
$$
\cHilb:  \ZZ\text{-schemes} \rightarrow \text{Sets}
$$
with
\begin{eqnarray*}
\cHilb_{\WP(\rho)}(T)&=& \left\{\text{\parbox{2.7in}{Isomorphism classes of  $ \cO_{\WP(\rho)_T}$-linear quotients 
		$\cO_{\WP(\rho)_T}\twoheadrightarrow \cQ,  $
 	    with  $\cQ$  flat over $T$}}\right \} \\[.5cm]
	&=& \left\{ \text{\parbox{2.7in}{Closed substacks $ \cX \subset \WP(\rho)_T
		$ over $ T, $ with $ \cX $ flat over $ T$}} \right\}.\\[.2cm]
\end{eqnarray*}
Then $\cHilb_{\WP(\rho)}$ is represented by a scheme (denoted also $\cHilb_{\WP(\rho)}$).

Furthermore, for each function $\fF: \ZZ \to \ZZ$, the subfunctor consisting of closed substacks $\cX \subset \WP(\rho)_T$ with Hilbert-Euler characteristic $$\chi(\cX_t, \cO_{\cX_t}(m)) = \fF(m)$$ for all geometric points $t\in T$ is represented by a {\em projective} scheme $\cHilb_{\WP(\rho),\fF}$ which is a  finite union of connected components of $\cHilb_{\WP(\rho)}$. 
\end{theorem}

\subsection{Cyclotomic stacks and line bundles}
Given an algebraic stack $\cX$,  
the {\em inertia stack} is an algebraic stack  whose objects are pairs $(\xi, \sigma)$, with $\xi$ an object of $\cX$ and $\sigma\in \Aut(\xi)$. 
The morphism $\cI_\cX\to \cX$ given by $(\xi,\sigma) \mapsto \xi$ is representable,
and $\cI_\cX$ is a group-scheme over $\cX$.  If $\cX$ is separated, the morphism $\cI_\cX \to \cX$ is proper.

\begin{definition}
A flat separated algebraic stack 
$$f:\cX \to B$$ locally of finite presentation over a scheme $B$ is said to be a {\em cyclotomic stack} if it has 
  cyclotomic stabilizers, i.e. if each geometric fiber of
$\cI_{\cX} \to \cX$
is isomorphic to the finite group scheme $\bmu_n$ for some $n$.
\end{definition}
Keeping $\cX$ separated and with finite diagonal, let $\pi:\cX \to X$ denote the coarse-moduli-space morphism;
it is proper and quasi-finite \cite{KeMo}.  The 
resulting family of coarse moduli spaces is denoted $\bar f:X \to B$.

\begin{example}
Let $Y\to B$ be an algebraic space, flat and separated over $B$.
If $\GG_m$ acts properly on $Y$ over $B$ with finite stabilizers then
the quotient stack $\cX:=\left[Y/\GG_m\right] \to B$ is cyclotomic.  

Indeed, the stabilizer scheme
$$Z=\{(y,h):h\cdot y=y \} \subset Y\times_B \GG_m$$
is equivariant with $\GG_m$-action on $Y\times_B \GG_m$
$$g\cdot(y,h)\ \ =\ \ (gy,ghg^{-1})\ \ =\ \ (gy,h).$$
On taking quotients, we obtain
$$\cI_\cX\ \ \ = \ \ \ \left[Z \,/\, \GG_m\right],$$ which is a closed substack of $\ \left[(Y\times_B \GG_m)\,/ \,\GG_m \right]\ =\  \left[\cX\times_B \GG_m\right].$
\end{example}

\begin{definition}
We say that  a stack $\cX$ {\em has index $N$} if for each object  
$\xi\in \cX(T)$ over a scheme $T$, and each automorphism 
$a\in \Aut \xi$ we have $a^N \ =\ id$, and if $N$ is the minimal positive integer satisfying this condition.
\end{definition}

\begin{lemma}
Let
$$f:\cX \to B$$
be a stack of finite presentation over a scheme $B$ and having finite diagonal.
There is a positive integer $N$ such that  $\cX$ has index $N$.
\end{lemma}

\begin{proof} This is well known: 
we may assume $\cX$ is of finite type over a field $k$. Then the inertia stack
$\cI_\cX$
is of finite type as well.  The projection morphism 
$$\cI_\cX \to \cX$$
is finite, with fibers the automorphism group-schemes 
of the corresponding objects of $\cX$. 
The degree is therefore bounded. \end{proof}

We note that cyclotomic stacks are tame \cite[Theorem 3.2 (b)]{AOV}. The following lemma is a special case of \cite[Lemma 2.2.3]{AV2}  in the Deligne--Mumford case, and  \cite[Theorem 3.2 (d)]{AOV} in general:

\begin{lemma} 
Consider a cyclotomic stack 
$$f:\cX \to B$$
 and a geometric closed point
$\xi:\Spec K \to \cX$ with stabilizer $\bmu_n$.  Let $\bar\xi:\Spec K \to X$ be the corresponding point. 
Then, in a suitable \'etale neighborhood of $\bar\xi$, the stack
$\cX$ is a quotient of an affine
scheme by an action of $\bmu_n$.  
\end{lemma}

The following is a special case of \cite[Lemma 2.2.3]{AV2} in the tame Deligne--Mumford case, and \cite[Corollary 3.3]{AOV} in general.

\begin{lemma} \label{Lem:cms-flat}
Let $f:\cX \to B$ be a cyclotomic stack (assumed by definition flat).
\begin{enumerate}
\item{The coarse moduli space
$\bar f:X \to B$ is flat.}
\item{Formation of coarse moduli spaces
commutes with 
base extension, i.e., for any morphism of schemes $X'\to X$,
the coarse moduli space of $\cX\times_X X'$ is $X'$.  In particular,
the geometric fibers of $\bar{f}$ are coarse moduli spaces
of the corresponding fibers of $f$.}
\end{enumerate}
\end{lemma}

\begin{lemma}\label{Lem:bundle-descends} Let $\cX\to B$ be a cyclotomic stack of index $N$
and $\cL$ an invertible sheaf on $\cX$.  
Then there is an invertible sheaf $M$ on $X$ and 
an isomorphism 
$$\cL^N\simeq \pi^*M.$$ 
\end{lemma}
\begin{proof}  Each automorphism group acts trivially on the fibers of
$\cL^N$ and hence trivially on the sheaf $\cL^N$.  Write $\cX$ locally near a geometric point $\bar\xi$ of $X$ as $[U/\bmu_r]$ with $r|N$;
assume $\cL$ is trivial on $U$. 
Then the total space of $\cL^N$ is the quotient $[(U\times \bbA^1)/\bmu_r]$, but the action on the factor $\bbA^1$ is trivial, and the total space is  $[U/\bmu_r]\times \bbA^1 = \cX\times \bbA^1$. Its coarse moduli space in this neighborhood is $X \times \bbA^1$.
It follows
that $M = \pi_*\cL^N$ is an invertible sheaf on the
coarse moduli space and $\pi^*M \to \cL$ is an isomorphism. See also \cite[Lemma 2.9]{AOVtwisted} 
\end{proof}

\begin{definition}\label{Def:principal}
Let $\cL$ be an invertible sheaf over a stack $\cX$.  
The stack
$$\cP_\cL \ \ :=\ \  \Spec_\cX \ \left(\oplus_{i\in \ZZ} \cL^i \right)\ \ \rightarrow\ \  \cX$$
is called the {\em principal bundle} of $\cL$.  
The grading induces a $\GG_m$-action on $\cP_{\cL}$ over $\cX$,
which gives $\cP_{\cL}$ the structure of a $\GG_m$-principal
bundle over $\cX$.  
\end{definition}
This definition requires the spectrum of a finitely-generated quasi-coherent 
sheaf of algebras, whose existence is guaranteed by \cite{LM} \S 14.2.

We shall require the following, see \cite[Lemma 4.4.3]{AV2} or \cite{LM}.
\begin{lemma} \label{lemma:repcrit}
Let $\cX$ and $\cY$ be algebraic stacks.  Let 
$g:\cX \to \cY$ be a morphism.  The following are equivalent:
\begin{enumerate}
\item{$g$ is representable;}
\item{for any algebraically closed field $K$ and 
geometric point $\xi:\Spec(K)\to \cX$,
the natural homomorphism of group schemes
$$\Aut \xi \to \Aut g(\xi)$$
is a monomorphism.}
\end{enumerate}
\end{lemma}
\begin{proof}
 Recall that an algebraic space is a stack with trivial stabilizers. Also note that if $V$ is an algebraic space and $\phi: V \to \cY$ a morphism, then the automorphisms of a point $\xi$ of $\cX\times _\cY V$ are precisely the kernel of the map $\Aut \phi(\xi) \to \Aut g(\phi(\xi))$. So if the condition on automorphisms holds then $\cX\times _\cY V$ has trivial stabilizers and therefore $g$ representable. 
Conversely, if $g$ is representable and $V \to \cY$ is
surjective then $\cX\times _\cY V$ is representable and $\cX\times _\cY V \to \cX$  surjective. So since $\cX\times _\cY V$ has trivial stabilizers we have $\Aut \phi(\xi) \to \Aut g(\phi(\xi))$ injective for all $\xi$.
\end{proof}

We can now state a result describing generating sheaves coming from powers of a line bundle.

\begin{proposition}
Let $\cX\to B$ be a cyclotomic stack of index $N$ and $\cL$
an invertible sheaf on $\cX$.  
The following conditions are equivalent.
\begin{enumerate}
\item $\cP_\cL$ is representable;
\item  the classifying morphism $\cX \to \cB \GG_m$ associated to $\cP_\cL$ is representable;
\item $\oplus_{i=0}^{N-1} \cL^{\otimes -i}$ is a generating sheaf for $\cX$.
\item for each geometric point $\xi: \Spec K \to \cX$ the action of $\Aut \xi$ on the fiber of $\cL$ is effective;
\end{enumerate}
\end{proposition} 
\begin{definition}
When any of the conditions in the proposition is satisfied, $\cX\to B$ is said to be {\em uniformized by $\cL$}, and $\cL$ is a {\em uniformizing line bundle} for $\cX\to B$.  
\end{definition}
\begin{proof}[Proof of Proposition] We first show that the first two conditions are equivalent. 
The representable morphism  
$$\Spec \ZZ \to \cB\GG_m$$ 
is the universal principal bundle, thus  
$$\cP_\cL = \cX\times_{\cB \GG_m}\Spec \ZZ.$$ 
Clearly if $\cX \to \cB \GG_m$ is representable, so is $\cP_\cL$. 
Conversely, if $\cP_\cL$ is representable, the automorphism group of any geometric 
object $\Spec K \to \cX$ of $\cX$ acts on the $\GG_m$-torsor 
$\Spec K \times_\cX \cP_\cL$ with trivial stabilizers;  in particular, it acts effectively. 
This means that the morphism $\cX \to \cB \GG_m$ induces a monomorphism on automorphism groups, 
and by Lemma~\ref{lemma:repcrit} it is representable.

Recall from \cite{OS}, Proposition 5.2 that a sheaf $\cF$ is a generating sheaf for $\cX$ if and only if for every geometric point $\xi:\Spec(K) \to \cX$ and every irreducible representation $V$ of $\Aut (\xi)$, the representation $V$ occurs in $\cF_\xi$. The third and fourth conditions are therefore equivalent.

We now show that the third condition is equivalent to the first. 

Assume the third condition, and fix a geometric point $\xi:\Spec(K) \to \cX$.  Then every irreducible representation of $\bmu_r$ occurs in the fiber $\oplus_{i=0}^N \cL^i_\xi$. This implies that the character $\cL_\xi$ is a generator of $\ZZ/r\ZZ$, and it clearly acts freely on the principal bundle of $\cL$. Conversely, if $\cX\stackrel{\cL}{\to}\cB\GG_m$ is representable, then $\bmu_r = \Aut  (\xi)$ injects into the automorphism group $\GG_m$ of  $\cL_\xi$, therefore the character of $\bmu_r$ in $\cL_\xi$ is a generator the dual group $\ZZ/r\ZZ$, and therefore all characters occur in $\oplus_{i=0}^r \cL^{-i}_\xi\subset \oplus_{i=0}^N \cL^{-i}_\xi$. \end{proof}

\subsection{Embedding of a cyclotomic stack in a weighted projective bundle}

\begin{definition}
Let $f:\cX \to B$ be a proper cyclotomic stack, with moduli space $\bar f:X\to B$. A {\em polarizing line bundle} $\cL$ on $\cX$ is a uniformizing line bundle, such that there is an $\bar f$-{\em ample} invertible sheaf $M$ on $X$ , an integer $N$ and an isomorphism
 $$\cL^N\ \ \simeq \ \ \pi^*M.$$\end{definition}

We have the following analogue of the standard properties of ample bundles:

\begin{proposition}\label{Prop:criterion}
let $\cX\to B$ be a proper family of stacks over a base of finite type, and let $\cL$ be a line bundle. The following conditions are equivalent:
\begin{enumerate}
\item The stack $\cX$ is cyclotomic,  and $\cL$ is a polarizing line bundle.
\item For every coherent sheaf $\cF$ on $\cX$ there exist integers $n_0$ and $N$ such that for any $n\geq n_0$ the sheaf homomorphism 
$$ \bigoplus_{j=n}^{n+N-1}H^0\left(\cX, \cL^j \otimes \cF \right)\otimes \cL^{-j}\ \  \to\ \  \cF$$ is surjective.
\item The stack $\cX$ is cyclotomic, uniformized by $\cL$, and for every coherent sheaf $\cF$ on $\cX$ there exists an integer $n_0$ such that for any $n\geq n_0$ and any $i\geq 1$ we have $$H^i(\cX,\cL^n\otimes \cF)=0.$$ 
\end{enumerate}
  
\end{proposition}

\begin{proof}
 Assume the first condition and let $\cF$ be a coherent sheaf. Let $N$ be the index of $\cX$. Fix $k$ such that $\cL^k = \pi^*M$ with $M$ ample and consider the sheaves $\cG_i=\pi_*(\cL^i\otimes\cF), i=0,\ldots N-1$ on $M$. There exists $m_0$ such that for all $m\geq m_0$ the sheaf homomorphisms 
$H^0\left(\cX, M^m\otimes \cG_i \right)\otimes \cO_X \to \cG_i$  are surjective. Note that $M^m\otimes \cG_i \simeq \pi_*(\cL^{mk+i}\otimes\cF)$. Also by Proposition \ref{Prop:criterion} we have that $\pi^*\pi_*\left(\oplus_{i=0}^{N-1} \cL^i\otimes \cF\right) \otimes \left(\oplus_{i=0}^{N-1} \cL^{-i}\right) \to \cF$ is surjective. The second condition follows. The reverse implication also follows: $\cL$ is uniformizing since the sheaf $\oplus_{i=-N-1}^0 \cL^i$ is necessarily generating by Proposition \ref{Prop:criterion}, and $M$ ample by Serre's criterion.

Since $\cX$ is tame we have that $\pi_*$ is exact on coherent sheaves. Writing $n=mk+j$ we have $H^i(\cX,\cL^n\otimes \cF)=H^i(X,\pi_*(\cL^n\otimes \cF)) = H^i(X,M^m\otimes \pi_*(\cL^i\otimes \cF))$. The equivalence of the first two conditions with the third  follows from Serre's criterion for ampleness applied to $M$. 
\end{proof}

\begin{proposition} \label{Prop:generate} Let $f:\cX \to B$ be as above, and $\cL$ a polarizing line bundle. There are positive integers $n<m$ such that 
\begin{enumerate}
\item for every geometric point $s\in B$ with residue field $K$,   we have $$H^i(\cX_s, \cL^j)\ =\ 0 
\ \ \forall \ i>0, \ j\geq n;$$
\item the sheaf $f_* \cL^j$ is locally free $\forall j\geq n;$
\item for every geometric point $s\in B$, the space $$\bigoplus_{j=n}^m \ H^0(\cX_s, \cL^j)$$ generates the algebra
$$K\ \  \oplus \ \ \bigoplus_{j\geq n} \ H^0(\cX_s, \cL^j).$$ 
\end{enumerate}
\end{proposition}

\begin{proof} Part (1) is a special case of part 3 of Proposition \ref{Prop:criterion}.

Part (2) follows from (1) by Grothendieck's ``cohomology and base change" applied directly to the $f$-flat sheaves $\cL^j$ since $H^1(\cX_s, \cL^j)=0$. Alternatively, consider the sheaves $\pi_*\cL^j$. These sheaves are flat over $B$ since $\cX$ is tame: since $\cL^j$ is flat we have that for any $\cO_B$-ideal $I$ the induced homomorphism $I\otimes_{\cO_B}\cL^j \to \cL^j$ injective. Since $\cX$ is tame we have $\pi_*$  exact, therefore   $I\otimes_{\cO_B}\pi_*\cL^j \to \pi_*\cL^j$ is  injective, as needed.
Since $H^1(X_s, \pi_*\cL^j)=0$ we can apply ``cohomology and base change'' to $X \to B$.

For (3), we first note that a range $n\leq j \leq m$ as required exists for any point $s$ of $B$: for each $i=0,\ldots, N-1$ there exists an $l$ such that 
$$H^0(\cX, \pi^*M^l\otimes \cL^i)\otimes H^0(\cX, \pi^*M^r) \to  H^0(\cX, \pi^*M^{l+r}\otimes \cL^i)$$
is surjective for all $r$. We may choose a single $l$ that works for all $i=0,\ldots, N-1$ and take $n=lN$.  

If at the same time we choose $n$ so that part (1) holds,  then elements of $H^0(\cX_s, \cL^j)$ in the given finite range lift to a neighborhood of $s$. Also the algebra $H^0(X, \oplus_{l \ge 0} M^l)$ 
is finitely generated and for each $j$, $H^0(X, (\pi_* \cL^j)\oplus M^l)$  is a finite module over the algebra. So a given range for $X_s$ works for a neighborhood of $s$.  Since $B$ is Noetherian, finitely many such neighborhoods cover $B$ and the maximal choice of $m$  works over all $B$. 
\end{proof}

\begin{corollary}\label{Cor:embed}
Let $n,m$ be as in Proposition \ref{Prop:generate}. Consider the locally free sheaves $W_j = f_*   \cL^j$. Then we have a closed embedding
$$\cX \hookrightarrow \WP\big( \bigoplus_{j=n}^m \ W_j  \big).$$
\end{corollary}
\begin{proof} Write 
$$\cR_{n,m} \ \ :=\ \  \Sym_B\big(\bigoplus_{j= n}^m \
W_j\big),$$ and denote $\bbA := \Spec_B\cR_{n,m}$.    
We have a $\GG_m$-equivariant map of $$\cP_\cL \to \bbA,$$ and since
$\cL^N = M$ is ample, the image is disjoint from the zero section
$0 \subset \bbA$. In
order to show that the morphism of quotient stacks 
$$ [\cP_\cL / \GG_m ] = \cX \ \ \to \ \ \big[ \big(\bbA \setmin  0\big)
  / \GG_m \big]$$
is an embedding, it suffices to show that the morphism $\cP_\cL \to \bbA$
  is an embedding. 

Let $$\cR_{n} = \cO_B \oplus \bigoplus_{j\geq n} W_j.$$ We have a surjective
algebra homomorphism $\cR_{n,m} \to \cR_n$, implying that 
$$\Spec_B \cR_n
\to \bbA$$ 
is a closed embedding. It suffices to show that $\cP_\cL \to
\Spec_B \cR_n$ is an embedding.

Next, let $\cR_\cL = \cR_0 = \bigoplus_{j\geq 0} W_j.$ The morphism $$\Spec_B \cR_\cL
\to \Spec_B \cR_n$$ induced by the natural inclusion $\cR_n\subset \cR_\cL$ is finite and birational, having its conductor supported
along the zero section. It therefore suffices to show that the
morphism $\cP_\cL \to \Spec_B \cR_\cL$ is an embedding. 

We now use the finite morphism $\cP_{\cL}\to \cP_M$, given by the finite ring
extension $\oplus_{l \in \ZZ} M^l \hookrightarrow \oplus_{j \in \ZZ} \pi_*\cL^j$.

Since $M$ is ample, we have an open embedding $\cP_M \hookrightarrow
\Spec_B \cR_M$. Taking the normalization $\hat\cP_{\cL}$ of $\Spec_B \cR_M$ in
the structure sheaf of $\cP_\cL$ we get an embedding $\cP_\cL \to
\hat\cP_{\cL}$. Note that  $\hat\cP_{\cL}$ is affine over $B$. Now sections
over affines in $B$ of the structure sheaf of $\hat\cP_{\cL}$ restrict to sections of
$\pi_*\cL^j$, in other words they come from global sections of
$\pi_*\cL^j$. It follows that $\hat\cP_{\cL} = \Spec_B \oplus_{j\in \ZZ}
f_*\cL^j,$ and since these vanish for $j<0$ we get $\hat\cP_{\cL} =\Spec_B
\cR_\cL$. \end{proof}

\section{Moduli of stacks with polarizing line bundles}\label{Sec:StaL}

\subsection{The category $\StaL$}
We are now poised to define a moduli stack. In order to avoid cumbersome terminology and convoluted statements, we will define the objects of the stack over schemes of finite type.
The extension to arbitrary schemes is standard but less than illuminating.

\begin{definition} We define a category $\StaL$, fibered in groupoids over the category of schemes, as follows:
\begin{enumerate}
\item An object $\StaL(B)$ over scheme of finite type
consists of a proper family $\cX \to B$ of cyclotomic stacks, with polarizing line bundle $\cL$.
\item An arrow from $(\cX\to B, \cL)$ to $(\cX_1 \to B_1 , \cL_1)$ consists of a fiber diagram 
$$\xymatrix {
\cX \ar^\phi[r]\ar[d] & \cX_1 \ar[d]\\ 
B\ar[r]              & B_1
}$$
along with an isomorphism $\alpha:\cL \to \phi^*\cL_1$.
\end{enumerate}
\end{definition}

\begin{remark} 
As defined, this is really a 2-category, because arrows of stacks in general have automorphisms, but because $\cL$ is polarizing  it is easy to see (see, e.g. \cite{AGV05}, \S 3.3.2) that it is equivalent to the associated category, whose morphisms are isomorphism classes of 1-morphisms in the 2-category. 
We can realize this category directly using the principal bundle $\cP_\cL$ as follows: 
an object of $\StaL(B)$ consists of  a $\GG_m$-scheme $\cP \to B$,  such that $\cX:=[\cP/\GG_m]$ is a proper cyclotomic stack polarized by the line bundle associated to $\cP$; 
an arrow is a $\GG_m$-equivariant fiber diagram
$$\xymatrix{
\cP\ar[r]\ar[d] & \cP_1 \ar[d]\\
B \ar[r]& B_1.
}$$ 
The fact that $\StaL$ parametrizes schemes with extra structure can be used to show that it is a stack. We will however show that it is an algebraic stack  via a different route - as suggested by the theorem below.

It is also worth noting that $\StaL$ is highly non-separated and far from finite type. 
\end{remark}  

\begin{remark} Note that the group $\GG_m$ acts by automorphisms on every object.
\end{remark}

\begin{theorem}\label{Th:PP-algebraic}
The category $\StaL$ is an algebraic stack, locally of finite type over $\ZZ$.
In fact, $\StaL$ has an open covering by $\sqcup \cQ_\alpha \to \StaL$, where
\begin{enumerate}
\item  each $\cQ_\alpha = [H_\alpha/ G_\alpha]$ is a global quotient, 
\item $H_\alpha \subset \cHilb_{\WP(\rho_\alpha)}$ is a quasi-projective open subscheme (described explicitly below), and
\item $G_\alpha = \Aut_{\rho_\alpha}(V)$, where $V$ is the representation space of $\rho_\alpha$.
\end{enumerate}
\end{theorem}

This theorem is a direct generalization of its well-known analogue, due  to Grothendieck, in the context of moduli of projective schemes endowed with an ample sheaf. The proof encompasses the rest of this section.

\subsection{Normally embeddable stacks}
Before we start the proof in earnest, we give some preliminary results.

\begin{definition} A stack $\cX$ uniformized by $\cL$ is {\em normally embeddable} in $\WP(\rho)$ if 
there exists a normal embedding $\cX \subset \WP(\rho)$ such that $\cO_\cX(1) \simeq \cL$.
\end{definition}

\begin{proposition} Fix a function $\fF: \ZZ \to \ZZ$ and two
positive integers $n<m$. For $n\leq i \leq m$ fix free modules $V_i$ of rank 
$f(i)$ 
and consider the natural representation  $\rho$ of $\GG_m$ on 
$V = \oplus_{i=n}^m V_i$.  
\begin{enumerate}
\item There is an open subscheme $H \subset \cHilb_{\WP(\rho),\fF}$ parametrizing normally embedded substacks. 
\item The subcategory $\StaL_\fF(\rho) \subset \StaL$ of objects with fibers normally embeddable in $\WP(\rho)$ and Hilbert--Euler characteristic $\chi(\cX, \cL^m) = \fF(m)$ satisfies $$\StaL_\fF(\rho) \simeq [H / G],$$ 
with $G = \Aut_\rho(V)$.
\end{enumerate}
\end{proposition}

\begin{remark}
Note that $G$ is the group of automorphisms of the free module representation rather than automorphisms of its projectivization. This accounts for the fact that $\GG_m$ acts on all objects of the quotient $[H / G]$.
\end{remark}

\begin{proof}[Proof of Proposition] For (1), note that each of the conditions in definition \ref{Def:normally} is an open condition, by the theorem on cohomology and base change.

We prove (2). The data of a map to the quotient stack $B\to[H/G]$ is equivalent to a family $\cY \to Q \to B$ where $Q$ is a principal $G$-bundle over $B$ and $\cY\subset \WP(\rho)\times Q$ is a family of normally embedded substacks.  Recall (see Definition \ref{Def:principal}) the principal $\GG_m$-bundle $\cP_{\cO_\cY(1)}\to \cY$ associated
to $\cO_\cY(1)$.    The free action of $G$ lifts canonically to $\cO_Y(1)$ and
to its principal bundle, and commutes with the $\GG_m$ action on the line bundle  by scaling.    To show that $\cY\to Q$ descends to $B$, consider the quotient $\cP_{\cO_\cY(1)}/G$.  This
retains a $\GG_m$ action, with finite stabilizers, whose quotient is
denoted $\cX$.  Let $\cL$ denote the line bundle associated to the
principal $\GG_m$-bundle $\cP_{\cO_\cY(1)}/G \to \cX$.   While we have scrupulously avoided the subtle procedure of taking the quotient of a stack under a free group
action,  we can identify $\cX = \cY / G$ and 
$\cL =  {\cO_\cY(1)}/G$, namely $\cX$ parametrizes principal $G$ bundles carrying an object of $\cY$ and similarly for $\cL$. A diagram summarizing the objects in this construction, where the parallelograms are cartesian squares, is as follows:
$$
\xymatrix{ {\ \cP_{\cO_\cY(1)}\ } \ar@{^{(}->}[r] \ar[rd]\ar[d]& {\ \cO_\cY(1)\ } \ar@{^{(}->}[r]\ar[d] & \cO(1) \times \cQ\ar[r]\ar[d] & \cO(1) \times \cH\ar[d] \\
\cP_\cL \ar[dr] & {\ \cY\ }   \ar@{^{(}->}[r] \ar[rd]\ar[d] & \WP(\rho) \times \cQ\ar[r]\ar[d] &  \WP(\rho) \times \cH\ar[d] \\
& \cX \ar[dr] & \cQ\ar[r]\ar[d] & \cH\ar[d] \\
&& B \ar[r]& [H/G].
}
$$
Now $\cX\to B$ is a cyclotomic stack polarized by $\cL$, with geometric fibers embeddable in $\WP(\rho)$: indeed the geometric fibers of $\cX \to B$ are isomorphic to those of $\cY\to \cQ$, which are cyclotomic, polarized by $\cO_\cY(1)$, and embedded in $\WP(\rho)$. The whole  construction works over arrows $B'\to B\to [H/G]$, and therefore gives a functor $[H/G] \to \StaL_\fF(\rho)$.

It is crucial in this construction that we work with the quotient of $H$ by the linear group $G$ and not its projective version - otherwise we would not have a functorial construction of the line bundle $\cL$.

We now consider the opposite direction. Let $\cX \to B$ be a cyclotomic stack polarized by $\cL$. Consider the locally-free sheaf $\oplus_n^m W_j$, with $W_j = f_* \cL^j$. Each fiber of $W$ is isomorphic as  $\GG_m$-space to a fixed free module $V$ with representation $\rho$. By the embeddability assumption and Corollary \ref{Cor:embed}, we have an embedding $\cX \hookrightarrow \WP(W)$.

Consider the principal $G$-bundle $\cQ = \Isom_{\GG_m} (W, V_B)$. We have a canonical isomorphism $W \times_B \cQ \simeq V_B$, giving an isomorphism $\WP(W) \times_B\cQ \simeq \WP(\rho)\times \cQ$. 
 Write $\cY = \cX \times_B\cQ$. We have a canonical induced  embedding $\cY  \hookrightarrow  \WP(\rho)\times \cQ$. This is a normal embedding by definition, and it is clearly $G$-equivariant. We thus obtain an object of $[H/G]$ as required.

Again the fact that the construction is canonical gives a construction for arrows in $B' \to B \to \StaL_\fF(\rho)$.  The proof that the two functors are inverse to each other is standard and left to the reader.  
 \end{proof}

We now consider what happens when we change the range of integers $n\leq i \leq m$.  

\begin{proposition} \label{Prop:change-n-m} Fix again a function $\fF: \ZZ \to \ZZ$ and
free modules $V_i$ of rank $\fF(i)$. For positive integers $n<m$ we again consider the natural representation  $\rho_{n,m}$ of $\GG_m$ on 
$V_{n,m} = \oplus_{i=n}^m V_i$.  Then
\begin{enumerate}
\item  $\StaL_\fF(\rho_{n,m}) \to \StaL_\fF(\rho_{n, m+1}) $ is an open embedding.
\item For any $n<m$ there exists a canonical open embedding   
$$\StaL_\fF(\rho_{n,m}) \ \ {\hookrightarrow} \ \ \StaL_\fF(\rho_{n+1,m+n}).$$ 
\item For any $n_1<m_1$ and $n_2<m_2$, there exist an integer $m\geq \max(m_1,m_2)$ and canonical open embeddings
$$\StaL_\fF(\rho_{n_1,m_1}) \hookrightarrow \StaL_\fF(\rho_{n, m}) $$ 
$$\StaL_\fF(\rho_{n_2,m_2}) \hookrightarrow \StaL_\fF(\rho_{n, m}) $$ 
with $n=\max(n_1,n_2)$.
\end{enumerate}
\end{proposition}

\begin{proof}[Proof of proposition] 
For (1), it is clear that we have an inclusion $\StaL_\fF(\rho_{n,m}) \hookrightarrow \StaL_\fF(\rho_{n, m+1})$. Given a family of graded rings generated in degrees $n,\ldots,m+1$, it is an open condition on the base that they are generated in degrees $n,\ldots,m$, as required. 

We prove (2). 
We have an open embedding $\StaL_\fF(\rho_{n,m}) \hookrightarrow \StaL_\fF(\rho_{n,m+n})$ by (1). In $\StaL_\fF(\rho_{n,m+n})$  we have an open substack $\StaL_\fF(\rho_{(n),n+1,m+n})$ of objects whose projection to the factors in degrees $n+1,\ldots, n+m$ is still an embedding. This also has an open embedding $\StaL_\fF(\rho_{(n),n+1,m+n})\hookrightarrow \StaL_\fF(\rho_{n+1,m+n})$  as it corresponds to objects such that $H^i(\cX,\cL^n)=0$ for $i>0$.
$$
\xymatrix{\StaL_\fF(\rho_{n,m}) \ar@{^{(}->}[r]\ar@{..>}[dr]& \StaL_\fF(\rho_{n,m+n})\\
      &   \StaL_\fF(\rho_{(n),n+1,m+n})\ar@{^{(}->}[r] \ar@{^{(}->}[u] &\StaL_\fF(\rho_{n+1,m+n})
}
$$

Now consider a  graded ring $R = \oplus R_i$. Suppose $R$ is generated in degrees $n,\ldots m$ over $R_0$. Then the truncation 
$$R' = R_0 \oplus \bigoplus_{i=n+1}^\infty R_i$$
 is generated in degrees $n+1,\ldots, n+m$. Indeed if $g_1\cdots g_k$ is a product of homogeneous terms in degrees between $n$ and $m$, of total degree $>n+m$, and is of minimal degree requiring a term $g_1$ of degree $n$ as a product of terms in the range $n,\ldots, m+n$,  then the other terms $g_2,\ldots, g_k$ are of degree between $n+1$ and $m$, so we may replace $g_1$ by $g_1'=g_1g_k$ and write $g=g_1'g_2\cdots g_{k-1}$ as a product of homogeneous terms in the range $n+1,\ldots,n+m$. 
 
This implies that the truncation of the ring of an object in   $\StaL_\fF(\rho_{n,m+n})$ coming from the open substack $\StaL_\fF(\rho_{n,m})$ is still normally generated. This means that the embedding $\StaL_\fF(\rho_{n,m}) \subset \StaL_\fF(\rho_{n,m+n})$ factors through the stack $\StaL_\fF(\rho_{(n),n+1,m+n})$.
 This induces the required open embedding $\StaL_{\mathfrak F}(\rho_{n,m}) \hookrightarrow \StaL_\fF(\rho_{n+1,m+n}),$ as required. 
  
 Part (3) follows by iterating the construction of part (2).
 \end{proof}

\subsection{Proof of Theorem \ref{Th:PP-algebraic}.} Given an object $f:\cX \to B$ in $\StaL(B)$, there is a maximal open and closed subset $B_0\subset B$ 
so that $f_0:\cX\times_B B_0\to B_0$ is an object in $\StaL_\fF(B_0)$;  
the formation of $B_0$ commutes with arbitrary base change. Thus we can decompose $$\StaL = \bigsqcup_\fF \StaL_\fF,$$ where $\StaL_\fF$ is the open and closed fibered subcategory parametrizing cyclotomic stacks polarized by a line bundle with Hilbert-Euler characteristic $\fF: \ZZ \to \ZZ$. We can thus fix $\fF$ and focus on $\StaL_\fF$.

Proposition \ref{Prop:change-n-m} gives for each pair of integers $n<m$ a fibered subcategory $\StaL_\fF(\rho_{n,m})\subset \StaL_\fF$. This fibered subcategory is open: first note that the condition that the higher cohomologies of $\cL^i$ vanish for $n\leq i \leq m$ is open by cohomology and base change. Second, the condition that the ring be normally generated in these degrees (including vanishing of cohomology  in higher degrees) is open as well. These open embeddings are compatible with the embeddings in Proposition \ref{Prop:change-n-m}. It follows that $\cup \StaL_\fF(\rho_{n,m}) \subset \StaL_\fF$ is an open subcategory. But by Corollary \ref{Cor:embed} every geometric point of $\StaL_\fF$ is in $\cup \StaL_\fF(\rho_{n,m})$, which implies that $\cup \StaL_\fF(\rho_{n,m}) = \StaL_\fF$ as fibered categories. Since $\cup \StaL_\fF(\rho_{n,m})$ is a direct limit of open subcategories which are algebraic stacks, it is also an algebraic stack.
\qed

\section{Moduli of polarized orbispaces} \label{Sec:Orb}
\subsection{Orbispaces}

\begin{definition}
An {\em orbispace} $\cX$  is a separated stack with finite diagonal,
of finite type over a field, equidimensional, geometrically connected and reduced, and 
admitting a dense open $U\subset \cX$ where $U$ is an algebraic space. 
\end{definition}

The assumption that $\cX$ be connected is made mainly for convenience.

Given two flat families of orbispaces $\cX_i \to B_i$ there exists a natural notion of a 1-morphism of families between them, namely a {\em cartesian square} 
$$ \begin{array}{ccc} 
\cX_1 & \to & \cX_2 \\ \dar & & \dar \\ B_1 & \to & B_2.
\end{array}$$
There is a natural notion of 2-morphism, making families of orbispaces into a 2-category.  However by \cite{AV1}, Lemma 4.2.3, such 2-morphisms are unique when they exist. It follows that this 2-category is equivalent to the associated category, where morphisms consist of isomorphism classes of 1-morphisms.
We call this the {\em category of families of orbispaces}. 

As an example, we have that the weighted projective stack $\WP(\rho_1,\ldots,\rho_r)$ is an orbispace if and only if $\gcd(\rho_1,\ldots,\rho_r)=1$.

We can consider cyclotomic orbispaces with uniformizing line bundle.
Adding the conditions in $\StaL$ we can define another category - the subcategory $\OrbL\subset \StaL$ of proper cyclotomic orbispaces with  polarizing line bundles. We have

\begin{proposition}
The subcategory  $\OrbL\subset \StaL$ forms an open substack
\end{proposition}

\begin{proof} 
Let $\cX \to B$ be a family of cyclotomic stacks polarized by $\cL$. 
The locus where the fibers are reduced and connected is open in $B$: see \cite[12.2.1]{EGA-IV-3},  for ``no embedded points'' and then cohomology and base change for ``connected''.  
Since the inertia stack is finite over $\cX$, the same holds for the locus where the fibers have a dense open with trivial inertia.
\end{proof}

\subsection{Polarizations} There is an important distinction to be made  between a polarizing line bundle and a polarization - the difference between the data of a line bundle and an element of the Picard group. 

Let $(\cX_i\to B_i, \cL_i)$ be two families with polarizing line bundles  $\cL_i$; then a morphism   comes from a 1-morphism $f_\cX: \cX_1 \to \cX_2$ sitting in a cartesian diagram as above, together with an isomorphism $\alpha: \cL_1 \to f_\cX^* \cL_2$. But a morphism of polarizations should ignore the $\GG_m$-action on the line bundles. The issue is treated extensively in the literature. A procedure for removing  the redundant action, called {\em rigidification}, is treated in \cite{ACV}, \cite{Romagny}, \cite{AGV05}, \cite[Appendix A]{AOV}. This is foreshadowed by the appendix in \cite{Artin}. 

Going back to our families  with polarizing line bundles, the  hypothesis that there is at least one point in each fiber where inertia is trivial implies that $\GG_m$ is a subgroup of the center of $\Aut(\cX, \cL)$ for any object (when the generic stabilizer is nontrivial, $\GG_m$ need not act effectively). This allows us to rigidify the stack $\OrbL$ along the $\GG_m$-action.
  Following the notation of \cite{Romagny, AGV05,AOV} (but different from \cite{ACV}) we thus have an algebraic stack
$$\Orbl = \OrbL\thickslash\GG_m.$$ An object $(\cX \to B,\lambda)$ of $\Orbl(B)$ is a {\em polarized family of cyclotomic stacks} over $B$. These will be described below.

Recall that we have a presentation $\StaL = \cup [H_\alpha/G_\alpha]$, with $H_\alpha$ a subscheme of the  Hilbert scheme of the weighted projective stack $\WP(\rho_\alpha)$. Whenever   $\WP(\rho_\alpha)$ is an orbispace the group $\GG_m$ embeds naturally in the center of the group $G_\alpha$, and we can form the projective group $\PP G_\alpha = G_\alpha/\GG_m$. We can  extract from this a presentation of $\OrbL$ and of $\Orbl$ as follows:

\begin{proposition}\label{Prop:OrbL}
Let $H_\alpha^{orb}\subset H_\alpha$ be the open subscheme parametrizing embedded orbispaces. Then
\begin{enumerate} 
\item $\OrbL = \bigcup_\alpha [H_\alpha^{orb}/G_\alpha]$.
\item $\Orbl = \bigcup_\alpha [H_\alpha^{orb}/\PP G_\alpha]$.
\end{enumerate}
\end{proposition}
This follows directly from the construction in any of \cite{ACV}, \cite{AGV05}, \cite[Appendix A]{AOV}.

This construction induces a morphism of stacks
$$\OrbL \to \Orbl,$$
which gives $\OrbL$ the structure of a gerbe  banded 
by $\GG_m$ over $\Orbl$.   

The stack $\Orbl$ can be described as follows: an object $(\cX \to B, \lambda)$ of $\Orbl(B)$ is a family of cyclotomic stacks $\cX \to B$, and $\lambda$ given locally in the \'etale
topology on $B$ by a polarizing line bundle on $\cX$.  Moreover, these
`almost' descend to a polarizing line bundle on $\cX$, i.e., on the overlaps these line bundles differ by a line bundle coming from the base;
this determines a section $\lambda$ of $\Pic(\cX/B)$  called a {\em polarization} of $\cX$.  

The obstruction to existence of a line bundle on $\cX$ is exactly the Brauer class in $H^2_{\text{\'et}}(B, \GG_m)$ of the $\GG_m$-gerbe $B\times_{\Orbl} \OrbL \to B$. 
If the obstruction is trivial, then a polarizing line bundle $\cL$ exists. In this case two pairs $(\cX, \cL)$ and $(\cX, \cL')$ represent the same polarization if and only if there is a line bundle $M$ on $B$ such that $\cL \simeq \cL' \otimes f^*M$. 

\subsection{Comparison of polarizations} 
The construction of rigidification involves stackification of a pre-stack. In our case this is much simpler as we end up sheafifying a pre-sheaf. 
The underlying fact is the following well-known result (cf. \cite[Lemma 1.19]{Viehweg}):
\begin{proposition} Consider a proper family of stacks $f:\cX \to B$  and $\cL, \cL'$ line bundles on $\cX$. Assume the fibers of $\cX \to B$ are reduced and connected. 

There exists a locally closed subscheme $B^0 \subset B$ and a line bundle $M$ on $B_0$, over which there exists an isomorphism $\cL|_{B^0} \to \cL'|_{B^0}\otimes f^* M$, and is universal with respect to this property.
\end{proposition}

The following is the outcome for polarizations. 

\begin{corollary} Consider a proper family of {\em orbispaces} $\cX \to B$  and $\lambda, \lambda'$ polarizations on $\cX$. 
There exists a locally closed subscheme $B^0 \subset B$, over which there exists an isomorphism $\lambda_{B^0} \to \lambda'_{B^0}$, and is universal with respect to this property.
\end{corollary}

\begin{proof}
Fix an \'etale covering $C\to B$ and line bundles $\cL,\cL'$ on $\cX_C$ representing $\lambda_C, \lambda'_C$. By the proposition there is $C^0\subset C$ and a line bundle $M$ on $C^0$, over which there exists an isomorphism $\cL|_{C^0} \to \cL'|_{C^0}\otimes f_C^* M$, and is universal with respect to this property. Since it is universal, the two inverse images of $C^0$ in $C\times_BC$ coincide, therefore $C^0$ descends to $B_0\subset B$, and the data $\cL|_{C^0} \to \cL'|_{C^0}\otimes f_C^* M$ gives an isomorphism $\lambda_{B^0} \to \lambda'_{B^0}$ by definition.
\end{proof}

\begin{corollary}\label{Cor:locally-closed} Consider a proper family of {\em orbispaces} $\cX \to B$, $\cL$ a line bundle,   $\lambda'$ a polarization on $\cX$.

There exists a locally closed subscheme $B^0 \subset B$, over which there exists an isomorphism $\lambda_{B^0} \to \lambda'_{B^0}$, with $\lambda$ the polarization induced by $\cL$, and is universal with respect to this property.
\end{corollary}
\begin{proof}
This is immediate from the previous corollary.
\end{proof}

\subsection{Moduli of canonically polarized orbispaces} 
\begin{lemma} \label{lemm:cmgor}
Let $\cX \to B$ be a proper family of orbispaces.
There exists open subschemes $B_{\gor} \subset B_{\cm}\subset B$ where the fibers of $\cX \to B$ are Gorenstein and Cohen-Macaulay respectively.
\end{lemma}
\begin{proof}
Recall the following fact \cite[12.2.1]{EGA-IV-3}: 
Suppose $X\to B$ is a flat morphism of finite type with pure dimensional
fibers and $F$ is a coherent $\cO_X$-module flat over $B$.
Then for each $r\ge 0$ the locus
$$\{b \in B : F|_{X_b} \text{ is } S_r \}$$
is open in $B$.  Take $F=\cO_X$ and $r=\dim(X/B)$ to see that the locus
where the fibers are Cohen-Macaulay is open.

When the fibers of $X\to B$ are Cohen-Macaulay, the shifted
relative dualizing complex $\omega_{X/B}^\bullet[-n]$ is a sheaf $\omega_{X/B}$ 
that is invertible
precisely on the open subset where
the fibers are Gorenstein
(see \cite{Conrad}, Theorem 3.5.1, Corollary 3.5.2, and the subsequent discussion).  

It remains to check this analysis applies to the stack $\cX$.  In characteristic 0, note that the dualizing
sheaf is insensitive to \'etale localization \cite[Th. 4.4.4 and p. 214]{Conrad}, and thus
descends canonically to $\cX$.  In arbitrary characteristic one has to use smooth covers instead.  However,
it is a fundamental property of the dualizing complex that it behaves
well under smooth morphisms, i.e., if $g:Y\to Z$ is smooth and $Z\to B$ is flat then (up to shifts)
$$\omega_{Y/B}^{\bullet} = g^*\omega_{Z/B}^{\bullet} \otimes \omega_{Y/Z}.$$
Indeed, these formulas are fundamental tools in showing that the dualizing complex is well-defined
(see \cite[pp. 29-30, Th. 2.8.1 and 3.5.1]{Conrad}).  In particular, we can define $\omega_{\cX/B}^{\bullet}$
via smooth covers.  

It the particular case at hand with $\cX = [\cP/\GG_m]$, the smooth cover is given by $\cP$ and the descent datum is given by the $\GG_m$-equivariant structure on $\omega_\cP$.
(Duality for Artin stacks is developed in \cite{Nironi}.
However we will only need the existence of a unique $\omega_\cX$.)
\end{proof}

The lemma implies:

\begin{proposition}
There are open substacks $\Orbl_\gor\subset \Orbl_\cm \subset \Orbl$ parametrizing Gorenstein and Cohen-Macaulay polarized orbispaces.
\end{proposition}

For most applications to compactifications of moduli spaces these substacks should be sufficient (see Remark~\ref{rema:KoKo}). 
However, in some situations more general singularities might be needed:

\begin{definition} 
A family  $\cX \to B$ of orbispaces is {\em canonically polarized} if 
\begin{itemize}
\item
the fibers are pure-dimensional, satisfy Serre's condition $S_2$, and are Gorenstein in codimension one;  
\item
$\omega_{\cX/B}^\circ$, namely the component of the dualizing complex
$\omega_{\cX/B}^\bullet$ in degree $\dim B-\dim X$, is locally free and polarizing.
\end{itemize}
The {\em canonical polarization} on $\cX\to B$ is the polarization induced by $\omega_{\cX/B}^\circ$.  
\end{definition}

\begin{theorem}\label{Th:OrbOm}
There is a locally closed substack $\Orbo\subset\Orbl$ parametrizing canonically polarized orbispaces. 
\end{theorem}
\begin{proof}
Lemma~\ref{lemm:cmgor} guarantees the conditions on the fibers are open.  
The change-of-rings spectral sequence gives a homomorphism of sheaves 
$$\omega_{\cX/B}^{\circ}|_{\cX_b} \rightarrow \omega_{\cX_b}^{\circ}.$$
This is injective because the relative dualizing sheaf is invertible and $\cX_b$ is reduced; 
its cokernel is supported on a subset of $\cX_b$ of codimension two.  
The dualizing sheaf $\omega_{\cX_b}^\circ$ is automatically saturated 
(cf. \cite[1.6]{ReYPG}), so any injection from an invertible sheaf 
that is an isomorphism in codimension one is an isomorphism.  
Thus on the open locus where $\omega_{\cX/B}^{\circ}$ is invertible,
its formation commutes with arbitrary base change.
Applying Corollary \ref{Cor:locally-closed}, we obtain the theorem.
\end{proof}

We denote the intersection $\Orbo_\gor :=  \Orbo\cap \Orbl_\gor$, the substack of canonically polarized Gorenstein orbispaces. 

Note that over $\Orbo$, the universal polarization is represented by an invertible sheaf,
namely, the relative dualizing sheaf  of the universal family.  This gives us a lifting
$$\xymatrix{
    & \OrbL\ar[d] \\ 
 \Orbo\ar@{.>}[ur] \ar[r] & \Orbl.
}$$

Indeed we can describe $\Orbo$ or $\Orbo_\gor$ directly in terms of line bundles as follows: an object of $\Orbo$ is a triple $(\cX, \cL,\phi)$ where $\cX$ is an orbispace on which $\omega_{\cX/B}^\circ$ is invertible, $\cL$ a line bundle, and $\phi: \cL \to \omega_{\cX/B}^\circ$ an isomorphism. Arrows are fiber diagram as for $\Orbo$. But note that an arrow does not involve the choice of an arrow on $\omega_{\cX/B}^\circ$: such an  arrow is canonically given as the unique arrow on $\omega_{\cX/B}^\circ$ respecting the trace map $\bR^n\pi_*\omega_{\cX/B}^\circ\to \cO_B$.
The relative dualizing sheaf (and complex) does have $\GG_m$ acting as automorphisms, but
this automorphism group acts effectively on the trace map. 

\section{Koll\'ar families and stacks}  \label{Sec:Kollar}

\subsection{Reflexive sheaves, saturation and base change}

Let $X$ be a reduced scheme of finite type over a field, having pure dimension $d$, satisfying Serre's condition $S_2$. 
\begin{lemma}
\begin{enumerate}
\item
Let $F$ be a coherent sheaf  on $X$. 
Then the sheaf $F^* := \Hom(F , \cO_X)$ is $S_2$.
\item Suppose $\psi:F \to G$ is a morphism of $S_2$-sheaves which is an isomorphism on the complement of a closed subscheme of  codimension $\geq 2$. Then $\psi$ is an isomorphism.
\item Let $F$ be an $S_2$ coherent sheaf on $X$. Then the morphism 
$F \to F^{**}$ is an isomorphism.
\end{enumerate}
\end{lemma}

\begin{proof}
 The problems being local, we assume $X$ is affine. 
\begin{enumerate}
\item Let $\phi$ be a section of $ \Hom(F,\cO_X)$ over an open set $U$ with $\codim(X\setminus U) \geq 2$. Let $f$ be a section of $F$ on $X$ and consider   $\phi(f_U)\in H^0(U, \cO_U)$. Since $X$ is $S_2$ this extends uniquely to a regular function $g$ on $X$. We define $\tilde\phi(f) = g$.  
\item Let $g$ be a section of $\cG$. Let $U$ be an open subset with codimension $\geq 2$ complement on which $\psi$ is an isomorphism. Set $f_U = \psi^{-1}(g_U)$, then $f_U$ is uniquely the restriction of a section $f$ of $F$, and $f = \psi^{-1}g$.
\item Since $F$ and $F^{**}$ are $S_2$, the homomorphism $F \to F^{**}$ is an isomorphism.
\end{enumerate}
\end{proof}

\begin{definition} A coherent sheaf $F$ on  $X$ is said to be {\em reflexive} if the morphism $F  \to F^{**}$ is an isomorphism.
\end{definition}

\begin{definition}
Let $F$ be a coherent sheaf on $X$, and $n$ a {\em positive} integer. 
We define 
$$F^{[-n]} = \Hom(F^{\otimes n},\cO_X).$$
and
$$F^{[n]} = \Hom(F^{[-n]},\cO_X).$$
\end{definition}

Theorem 10 of \cite{kollarjalg} implies 
\begin{proposition}\label{Prop:flatsheaf}
Let $X\to B$ be a flat morphism of finite type with $S_2$ fibers of pure dimension $d$. Let $U\subset  X$ be  an open subscheme, dense in each fiber. Let $F$ be a coherent sheaf on $X$, locally free on $U$. If
the formation of $F^{[n]}$ commutes with arbitrary base extension
then $F^{[n]}$ is flat over $B$.  
\end{proposition}

\subsection{Koll\'ar families of $\QQ$-line bundles}  
\begin{definition} \label{Def:Kollar}
\begin{enumerate}
\item By {\em a Koll\'ar family of $\QQ$-line bundles} we mean
\begin{enumerate}
 \item $\bar f:X \to B$ a flat family  of equidimensional connected reduced schemes satisfying Serre's condition $S_2$,
\item  $F$ a coherent sheaf  on $X$, such that 
  \begin{enumerate}
    \item for each fiber $X_b$, the restriction $F|_{X_b}$ is reflexive of rank 1;
    \item  for every $n$, the formation of $F^{[n]}$ commutes with arbitrary base change;
    \item  for each geometric point $s$ of $B$ there is an integer $N_s\neq 0$ such that  $F^{[N_s]}|_{X_s}$ is invertible.
  \end{enumerate}   
\end{enumerate}
\item A {\em morphism} from a  Koll\'ar family $(X\to B, F)$ to another $(X_1\to B_1, F_1)$ consists of  a cartesian diagram
$$\xymatrix{
X \ar[d]\ar^{\bar \phi}[r] & X_1 \ar[d]\\
B\ar[r] & B_1 
}$$ along with an isomorphism $\bar\alpha:F \to \bar\phi^* F_1$.
\item Parts (1) and (2) above define the objects and arrows of a category  of {\em Koll\'ar families of  $\QQ$-line bundles}, fibered over the category $\cS ch_k$ of $k$-schemes.  It has an important open subcategory $\KolL$ of {\em Koll\'ar families of polarizing $\QQ$-line bundles},
where $X \to B$ is {\em proper} and $F^{[N_s]}|_{X_s}$ is {\em ample.}
\end{enumerate}
\end{definition}

Note that (1.b.i) and (1.b.ii) imply that $F = F^{[1]}$.
Furthermore, (1.b.i) and (1.b.iii) imply that $F$ is invertible in codimension 1. 
A reflexive rank-one sheaf on a reduced one-dimensional scheme $X$ is a fractional ideal $J$.
Indeed, express $\cF^*$ locally as a quotient $\cO_X^{\oplus r} \to F \to 0$,  which realizes 
$\cF$ as a subsheaf of $\cO_X^{\oplus r}$.  Choose a projection $\cO_X^{\oplus r} \to \cO_X$ such that
the composed homomorphism $\cF \to \cO_X$ has rank one at each generic point.  This
is injective since our scheme is reduced.
The saturated powers $F^{[n]}$ equal $J^n$, i.e., the powers of $J$ as an ideal.  
In particular, $J^N$ is locally principal.
The blow-up of $X$ along $J$ is
$$\Proj_X\oplus_{n\ge 0} J^n \simeq \Proj_X \oplus_{n\ge 0} J^{Nn} = X,$$
so $J$ is locally principal and $F$ is invertible.

One can show that the category of {\em proper} Koll\'ar families 
is an algebraic stack.
We will show below that the subcategory $\KolL$ is an algebraic stack by
identifying it as an algebraic substack of $\OrbL$.
Koll\'ar families were introduced in the canonical case $F = \omega_{X/B}$
in \cite[5.2]{KollarJDG} and \cite[2.11]{HK}.  
In \cite{Hacking} it was shown that they admit, at least in the 
canonical case, a good deformation-obstruction theory. 
As we see below, this holds in complete generality. Indeed Hacking's approach is via the associated stacks, as is ours. In \cite{KollarHH} Koll\'ar provides an approach without stacks.

We note that the index $N_s$ is bounded in suitable open sets:

\begin{lemma}\label{Lemma:bound}
Let $(X \to B,F) $ be a Koll\'ar  family of $\QQ$-line bundles, and $s\in B$ a geometric point. Let $N_s$ be an integer such that  $F^{[N_s]}|_{X_s}$ is invertible. Then there is an open neighborhood $U$ of $s$ such that $F^{[N_s]}|_{X_U}$ is invertible.
\end{lemma}

\begin{proof} Since  $F^{[N]}$ are flat and their formation commutes with base change, the assumption that $F^{[N_s]}|_{X_s}$ is invertible
implies $F^{[N_s]}$ is invertible in a neighborhood of $X_s$.
Since $X\to B$ is of finite type the assertion follows. \end{proof}

\subsection{Koll\'ar families and uniformized twisted varieties}
\begin{definition} Let $(\bar f:X \to B \ , \ F )$ be a Koll\'ar family of $\QQ$-line bundles. 
\begin{enumerate} 
 \item The {\em $\GG_m$-space} of $F$ is the $X$-scheme 
 $$ \cP_F\ \  =\ \  \Spec_X \left(\oplus_{j\in \ZZ} F^{[j]}\right).$$ 
 Note that $\cP_F \to B$ is flat by Proposition \ref{Prop:flatsheaf}.
 \item The {\em associated stack} is $$\cX_F\ \  =\ \   \left[\,\cP_F \ \big/ \ \GG_m\,\right].$$ This comes with the   {\em natural line bundle} $\cL = \cL_F$ associated to the principal bundle $$\cP_F\ \to\ \cX_F.$$
\end{enumerate}
\end{definition}
We drop the index  $F$ and write $\cX$ and $\cP$ when no confusion is likely. We denote by $\pi:\cX \to X$ the resulting morphism.

\begin{proposition}\label{Prop:KQ-to-PP}
\begin{enumerate}
\item The family $\cX \to B$ is a family of cyclotomic orbispaces uniformized by $\cL$, with fibers satisfying Serre's condition $S_2$. 
\item The morphism $\pi:\cX\to X$ makes $X$ into the coarse moduli space of $\cX$. This morphism  is an isomorphism on the open subset where $F$ is invertible, the complement of which has codimension $>1$ in each fiber. 
\item For any integer $a$, we have $\pi_*(\cL^a)=F^{[a]}$ (in particular $\pi_*(\cL^a)$ is reflexive). 
\item This construction is functorial, that is, given a morphism of Koll\'ar families $(\bar\phi,\bar \alpha)$ from $(X \to B, F)$ to  $(X_1 \to B_1, F_1)$ we have canonically $\cP_F \simeq \bar\phi^*\cP_{F_1}$ and $\cX_F \simeq \bar\phi^*\cX_{F_1}$.
\end{enumerate}
\end{proposition}

\begin{proof} To verify that $\cX \to B$ is cyclotomic and uniformized
by $\cL$, we just need to check that $\GG_m$ acts on $\cP_F$ with finite stabilizers.  Lemma \ref{Lemma:bound} allows us to assume
that $F^{[N]}$ is locally free for some $N>0$.  We have natural
homomorphisms 
$$m_a:\ (F^{[N]})^a\ \  \lrar\ \  F^{[Na]}$$
for each $a\in \ZZ$.   For each $b\in B$, the sheaf $F$ is locally free on an open set
$U_b\hookrightarrow X_b$ with codimension-two complement, 
so $m_a$ is an isomorphism over $U_b$, and hence over all $X_b$.  
It follows that $(F^{[N]})^a=F^{[Na]}$.  

Consider the $\GG_m$-equivariant morphism
$$\cP_F\ \  \lrar \ \  \Spec_X \left(\oplus_{a\in \ZZ} F^{[aN]}\right).$$
Note that $\GG_m$ acts on the target with stabilizer $\bmu_N$,
which implies the stabilizers on $\cP_F$ are subgroups of $\bmu_N$.   

To complete part (1), we note that a fiber in $\cP_F$ over $b\in B$ is $S_2$ as the spectrum of an algebra with reflexive components over an $S_2$ base. Also $\cP_F \to \cX$ is smooth and surjective. Smoothness implies
that  the quotient stack $\cX$ is $S_2$
as well.

Since the coarse moduli space is obtained by taking invariants (see the proof of Lemma \ref{Lem:cms-flat}), and since the invariant part of $\oplus F^{[a]}$ is  $F^{[0]}=\cO_X$, we have that $X$ is the coarse moduli space of $\cX$. Over the locus $U$ where $F$ is invertible, the scheme $\cP_F$ is a principal $\GG_m$-bundle, so $$\cX_U\  = \ \left[\,(\cP_F)|_U\ /\ \GG_m\,\right]\  =\  U.$$

Part (3) follows since $\cL^a$ is the degree-$a$ component of the graded $\cO_\cX$-algebra $\oplus \cL^i$, whose spectrum is $\cP_F$. The degree-$a$ component of the graded $\cO_X$ algebra $\oplus_{i\in \ZZ} F^{[i]}$ is $F^{[a]}$. 

Part (4) follows from the functoriality of $\Spec$ and of the formation of quotient stack.

\end{proof}

\begin{remark} If we assume for each $a\in \ZZ$ and $b\in B$ that  $F^{[a]}$ is $S_r$ (in particular the fibers of $X$ are $S_r$), it follows that the fibers of  $\cX$ are $S_r$ as well.
\end{remark}

\begin{definition} \label{Def:twisted-variety}
A family of  \emph{uniformized twisted varieties} $(\cX \to B, \cL)$ is a  flat
family of $S_2$ cyclotomic orbispaces uniformized by $\cL$, 
such that the morphism $\pi:\cX \to X$ to the coarse moduli space is an 
isomorphism away from a subset of codimension $>1$ in each fiber.  

We define 1-morphisms of families of uniformized twisted varieties as fibered diagrams. Of course, there is a notion of 2-morphism making these into a 2-category, but since these are orbispaces, a 2-isomorphism is unique when it exists \cite[Lemma 4.2.3]{AV2}, so this 2-category is equivalent to the associated category, whose morphisms are isomorphism classes of 1-morphisms.
\end{definition}

\begin{remark}
 We insist on $\pi:\cX \to X$ an isomorphism in codimension-1 to obtain an equivalence with Koll\'ar families as in Theorem \ref{Th:KQisPP} below. To describe uniformized orbispaces in terms of sheaves on $X$ in general, it is necessary to specify an algebra of sheaves, which can be quite delicate especially when $\cX\to X$ is branched along singular codimension-1 loci. For a subtle example see \cite[Definition 2.5]{Jarvis} as well as \cite[\S 4.2-4.4]{AJ}, where an equivalence analogous to our Theorem \ref{Th:KQisPP} below is given.
\end{remark}

\begin{theorem} \label{Th:KQisPP}
The category of Koll\'ar families of $\QQ$-line bundles is equivalent to the category of uniformized twisted varieties via the base preserving functors
$$ (X\to B, F) \ \ \ \mapsto \ \ \ (\cX_F  \to B, \cL_F), $$
with $\cX_F =  [\cP_F \ / \ \GG_m]$, 
and its inverse
$$ (\cX \to  B, \cL) \ \ \  \mapsto \ \ \ (X \to B, \pi_*\cL),$$
where $X$ is the coarse moduli space of $\cX$.
\end{theorem}
\begin{proof}
Proposition \ref{Prop:KQ-to-PP} gives the functor from the category of Koll\'ar families of $\QQ$-line bundles to the category of uniformized twisted 
varieties.  We now give an inverse.  Fix a uniformized
twisted variety $(\cX \to B,\cL)$, with coarse moduli space
$\pi:\cX \to X$.  Since the formation of $\pi$ commutes with arbitrary base change (Lemma \ref{Lem:cms-flat}), the universal property of coarse moduli spaces guarantees that each fiber of $X\to B$ is reduced.  If $N$ is the index of $\cX$ then, by Lemma \ref{Lem:bundle-descends}, $\cL^N$ descends to an invertible sheaf on $X$, which coincides with $\pi_*\cL^N$.  

Each geometric point of $\cX$ admits an \'etale neighborhood which
is isomorphic to a quotient of an affine $S_2$-scheme $V$ by the action of $\bmu_r$.  Locally, the coarse moduli space is the scheme-theoretic 
quotient $V/\bmu_r$.  However, since $\bmu_r$ is linearly reductive, any quotient of an $S_2$-scheme by $\bmu_r$ is also $S_2$, because the invariants are a direct summand in the coordinate ring of $V$. Similarly the sheaves $\pi_*\cL^j$ are direct summands in the algebra of $\cP$, which is affine over $X$, flat over $B$ with $S_2$ fibers. It follows that these sheaves are flat over $B$, saturated, and their formation commutes with base change on $B$.
\end{proof}

\begin{corollary}
The category $\KolL$ is an algebraic stack, isomorphic to the open substack of $\OrbL$ where $(\cX\to B, \cL)$  are uniformized twisted varieties. 
\end{corollary}
\begin{proof} 
 Using Theorem \ref{Th:KQisPP}  note that  $\KolL$ indeed parametrizes  orbispaces  with polarizing line bundles which are at the same time  uniformized twisted varieties. This is open in $\OrbL$ since the conditions of being $S_2$ is open (\cite[12.2.1]{EGA-IV-3}) and the condition on the fibers of inertia having support in codimension $>1$ is open  by semicontinuity of fiber dimensions in proper morphisms.
\end{proof}

\begin{definition} 
We define $\Koll = \KolL\thickslash \GG_m$, where $\GG_m$ acts by scalars on $\cL$. We define $\Kolo \subset \Koll$ as the locally closed subcategory corresponding to
canonically polarized twisted varieties $(f:\cX\to B, \lambda)$  where $\lambda$ is given by  $\omega_{\cX/B}^\circ$.

We define   $\Kolo_\gor \subset \Kolo$
to be the open substack where the fibers of $f$ are Cohen-Macaulay.
\end{definition}
Note that Corollary ~\ref{Cor:locally-closed}  guarantees that
the second condition is locally closed. 

\begin{remark}\label{Rem:CM} We can interpret $\Kolo$ in terms of Koll\'ar families $(\bar f:X\to B, F)$ with isomorphism $F\to\omega_{\cX/B}^\circ$. The substack $\Kolo_\gor$ is characterized by 
\begin{enumerate}
\item the fibers of $f$ are Cohen-Macaulay;
\item each power $\omega^{[n]}_{X/B}$ is Cohen-Macaulay.
\end{enumerate}
Indeed, if a sheaf $\cG$ on $\cX$ is Cohen-Macaulay then $\pi_*G$ is also Cohen-Macaulay: locally write $\cX$ as a quotient of an affine  scheme $V$ by $\bmu_r$, and $\cG$ corresponding to an equivariant module $G$. The complex computing local cohomology of $G$ can be taken $\bmu_r$-equivariant, and the invariant subcomplex splits. Therefore, if the local cohomology vanishes the invariant part, computing local cohomology of $G^{\bmu_r}$ on $X$ vanishes.

In general, the second condition is not a logical consequence of the first.
There are examples \cite[\S 6]{Singh} of Cohen-Macaulay log canonical threefolds whose
index-one covers are not Cohen-Macaulay.  
\end{remark}

\begin{remark}
In characteristic zero, varieties with canonical (or even log terminal) singularities
admit index-one covers with canonical singularities \cite{Re80Ca3folds}, which
are therefore rational \cite{Elkik} and Cohen-Macaulay \cite{KoMo}.
However, in positive characteristic there exist log terminal surface singularities
with index-one covers that are not canonical
\cite{KawamataPathology}.  Generally, index-one covers present technical difficulties
when the index of the singularity is divisible by the characteristic, especially
in small characteristics.  Our approach sidesteps these issues.  
\end{remark}

\section{Semilog canonical singularities and compactifications}\label{Sec:SLC}

From here on we assume our schemes are over a field of characteristic 0.

\subsection{Properness results and questions}\label{Sec:properness}
Recall that canonical 
singularities deform to canonical singularities \cite{KawCan}.  This fact and
Proposition~\ref{prop:SLC} of the appendix 
allow us to formulate the following:
\begin{definition}
We define $\Kolo_\can\subset \Kolo$ as the open substack
corresponding to
canonically polarized twisted varieties where {\em both} $\cX$ and $X$ have canonical singularities.  

We define $\Kolo_\slc, \Kolo_\gorslc \subset \Kolo$ as the open substack
corresponding to
canonically polarized twisted varieties where  $\cX$ has semilog canonical (resp. Gorenstein semilog canonical) singularities.   

Objects of $\Kolo_\slc$ are called {\em families of twisted stable varieties}.
\end{definition}
Theorem~\ref{Th:KQisPP} induces equivalences of the  categories above with the respective categories of Koll\'ar families and rigidified Koll\'ar families.

We review what is known about properness of moduli spaces
of stable varieties:
\begin{proposition}
Each connected component of the
closure of $\Kolo_\can$ in $\Kolo_\slc$ is proper with projective
coarse moduli space.  
\end{proposition}
\begin{proof}
Most of this is contained in \cite{Karu}, though our definition of
families differs from Karu's.

We first check that the closure of $\Kolo_\can$ satisfies the
valuative criterion for properness.

Let $\Delta$ be the spectrum of a discrete valuation ring, with
special point $s$ and generic point $\eta$.  Let $\cX_{\eta}$ be a
canonically-polarized orbispace with Gorenstein canonical 
singularities over $\eta$.
Let $X_{\eta}$ denote its coarse moduli space.  

Apply resolution of singularities and semistable reduction \cite{KKMS} to 
obtain
\begin{itemize}
\item a finite branched cover $\Delta'\to \Delta$ with generic point $\eta'$
and special point $s'$;
\item a nonsingular variety $Y$ and a flat projective morphism $\phi:Y\to \Delta'$,
with $Y_{s'}$ reduced simple normal crossings;
\item a birational morphism $Y_{\eta'} \to X_{\eta'}$.
\end{itemize}
Consider the canonical model of $Y$ relative to $\phi$
$$\psi:W \to \Delta',$$
which exists due to \cite[Theorem 1.1]{BCHM}.
It has the following properties:
\begin{itemize}
\item $K_{W/\Delta'}$ is $\QQ$-Cartier and ample relative to $\psi$;
\item $W$ has canonical singularities, so every exceptional divisor has
discrepancy $\ge 0$.
\end{itemize}
As a consequence, we deduce
$W_{\eta'}\simeq X_{\eta'}$ and
$W$ is Cohen-Macaulay.

Furthermore,
every exceptional divisor in the special fiber has log 
discrepancy $\ge 0$ with respect to $(W,W_{s'})$, thus
$W_s$ has normal crossings in codimension one.
In particular, the pair $(W,W_{s'})$ is log canonical.
Applying `Adjunction' \cite{kawakita} and 
criterion \ref{lemma:SLCcrit}, we conclude
that $W_{s'}$ has semilog canonical singularities.
The variety $W_{s'}$ is the desired stable limit of $X_{\eta}$;  it is unique by the 
uniqueness of the canonical models.

Now we construct our desired stack-theoretic stable limit. The point is that the stack-theoretic canonical model has canonical singularities for the same reason $W$ has. In classical terms,   
since $W$ has canonical singularities, its index-one cover does as well \cite[5.20,5.21]{KoMo}.
Furthermore, the direct-image sheaves
$$\psi_*\omega_{\psi}^{[n]}, \quad n\ge 0,$$
are locally free and commute with restriction to the fibers, i.e.,
$$\psi_*\omega_{\psi}^{[n]}|_{s'}=\Gamma(W_t,\omega_{W_{s'}}^{[n]}).$$
This formulation of `deformation invariance of plurigenera' 
follows from the existence of minimal models \cite[Cor. 3, Thm. 8]{nakayama}
and the good behavior of direct images of dualizing sheaves \cite[Thm 2.1]{KollarHDI}.

Since $W$ is a canonical model, we can express
$$W=\Proj_{\Delta'}  \oplus_{n \ge 0} \psi_*\omega_{\psi}^{[n]}.$$
The corresponding affine cone
$$CW:=\Spec_{\Delta'} \oplus _{n \ge 0 } \psi_*\omega_{\psi}^{[n]}$$
comes with a $\GG_m$-action associated with the grading, and we
define
$$\cW=\left[ \, (CW\setminus 0)\ /\ \GG_m\,\right].$$
Again, we have $\cW_{\eta'}\simeq \cX_{\eta'}$ (as polarized
orbispaces) and thus $\cW_{s'}$ has Gorenstein
semilog canonical singularities.  Indeed, $\cW$ has canonical
singularities (it is locally modeled by the index-one
cover of $W$) and thus is Cohen-Macaulay;  the dualizing
sheaf $\omega_{\cW/\Delta'}$ is invertible by the $\GG_m$-quotient
construction.  Repeating the previous adjunction
analysis yields that $\cW_{s'}$ is semilog canonical.

The uniqueness of the stack-theoretic limit
follows from
the uniqueness of the canonical model, via Theorem~\ref{Th:KQisPP}.

We turn now to boundedness:  \cite[\S 3]{Karu} still shows that 
each connected component of the closure of $\Kolo_\can$ is of finite
type:  Matsusaka's big theorem shows the open locus $\Kolo_\can$ is of finite type. We apply weak semistable reduction to an arbitrary compactification of a tautological family, and obtain a family $X\to B$ which is weakly semistable.  
In particular, $B$ is smooth.  Applying finite generation (\cite{BCHM}) we can  take the relative canonical model $X^\can=\Proj_B R(X/B) \to B$, or its orbispace version $\cX^\can=\cProj_B R(X/B) \to B$. The total space has canonical singularities as it is a canonical model - this is true for either $X^\can$ or $\cX^\can$. Deformation invariance of plurigenera shows that this is flat. An argument similar to the one below shows that the fibers are semilog canonical. This gives a compactified family over a base of finite type as required.

The argument there, citing Koll\'ar's projectivity
criterion \cite{KollarJDG}, shows that properness of the closure
of $\Kolo_\can$  implies projectivity of its coarse moduli space.  

\end{proof}

We would like the following formulation of `semi-canonical models':
\begin{assumption} \label{assume:semiMMP}
Suppose that $\Delta$ is the spectrum of a discrete valuation ring with 
special point $s$ and generic point $\eta$.  Let $X_{\eta}$ be a stable variety
over $\eta$.  Then there exists a finite ramified base-change 
$\Delta'\to \Delta$ with special point $s'$ and generic point $\eta'$, a unique variety $W$, and a flat morphism 
$\psi:W\to \Delta'$ satisfying the following:
\begin{itemize}
\item{$W_{\eta'}$ 
is isomorphic to the base-change  $X_{\eta'} = X \times_\eta \eta'$.}
\item{The relative dualizing sheaf $\omega_{\psi}$ is $\QQ$-Cartier
and ample.}
\item{
Each reflexive power $\omega^{[n]}_{\psi}$ is $S_2$ relative
to $\psi$, i.e., $S_2$ on restriction to the fibers.
If $X_{\eta}$ is Cohen-Macaulay, we expect the sheaves $\omega^{[n]}_{\psi}$ to be Cohen-Macaulay
as well. \cite{KolKov}}
\item{$W_{s'}$ has semilog canonical singularities.}
\item{The sheaves $\psi_*\omega_{\psi}^{[n]},n\ge 0$ are locally
free and satisfy
$$\psi_*\omega_{\psi}^{[n]}|_{s'}=\Gamma(W_{s'},\omega_{W_{s'}}^{[n]}).$$}
\end{itemize}
\end{assumption}

\begin{proposition} 
Under Assumption~\ref{assume:semiMMP}, $\Kolo_\slc$ satisfies the 
valuative criterion for properness.  In particular, each connected 
component of finite type is proper.
\end{proposition}
\begin{remark}
To say that each connected component of $\Kolo_\slc$ is of finite type
is a boundedness assertion about the stable varieties with given
numerical invariants.  These assertions have been proven in dimensions
$\le 2$ \cite{Alexeev}, but remain open in higher dimensions.
\end{remark}

\begin{remark} \label{rema:KoKo}
The discussion following Corollary 1.3 of \cite{KolKov} implies that 
the stable limit
of a one-parameter family of Cohen-Macaulay stable 
varieties (resp. Gorenstein twisted stable varieties) is Cohen-Macaulay
(resp. Gorenstein).  
Thus the $\Kolo_{\gorslc} \subset \Kolo_\slc$ is a 
union of connected components. 

Nevertheless, there are good reasons to consider non-Cohen-Macaulay
singularities. We have already mentioned non-Cohen-Macaulay log canonical singularities in Remark \ref{Rem:CM}. In characteristic zero, log terminal singularities
and their index-one covers are automatically Cohen-Macaulay \cite{KoMo}.
\end{remark}

\appendix 
\section{The semilog canonical locus is open}
\subsection{Statement}
The following result is known for surfaces (\cite{K-SB,KollarJDG}). The present proof was suggested to us by V. Alexeev. The second part can be found in \cite{Karu}.

\begin{proposition} \label{prop:SLC}
Let $B$ be a scheme of finite type over a field
of characteristic zero and
$f:X \to B$ be a flat morphism of finite type
with connected reduced equidimensional 
fibers.  Assume that the fibers are $S_2$ and Gorenstein
in codimension one.  

Suppose there exists an invertible sheaf $\cL$ on $X$ such that
for each $b \in B$ we have
$$\cL|X_b\simeq j_*\omega^{\otimes n}_U, \quad n\in \NN,$$
where $j:U\hookrightarrow X_b$ is the open subset where $X_b$ is
Gorenstein.  
Then the locus where the 
fibers have semilog canonical singularities
$$B^{slc}=\{b\in B: X_b \text{ is semilog canonical} \}
$$
is open in $B$.
\end{proposition}

\begin{remark}
In the case where the fibers are Gorenstein, the existence of
$\cL$ is automatic.  Indeed, the relative dualizing sheaf 
$\omega_{X/B}$ is itself invertible.
\end{remark}

Before starting the proof proper, we recall a 
characterization of semilog canonical singularities,
which can be found in chapter 2 of 
\cite{FlAb}:
\begin{lemma} \label{lemma:SLCcrit}
Let $(Y,D)$ be a pair consisting of a
connected equidimensional $S_2$ scheme $Y$ and a
reduced effective divisor $D$
with no components 
contained in the singular locus of $Y$.
Suppose that $Y$ has normal crossings in codimension
one and that the Weil divisor
$K_Y+D$ is $\QQ$- Cartier.  
Let $\nu:\tilde{Y} \to Y$ denote the 
normalization and $C\subset \tilde{Y}$ its conductor,
i.e., the reduced effective divisor such that
$$\nu^*(K_Y+D)=K_{\tilde{Y}}+\tilde{D}+C,$$
where $\tilde{D}$ is the proper transform of $D$.  
Then the following conditions
are equivalent:
\begin{itemize}
\item{$(Y,D)$ is semilog canonical.}
\item{$(\tilde{Y},\tilde{D}+C)$ is log canonical.}
\end{itemize}
\end{lemma}

The condition of having normal crossings in codimension
1 is also open.  Indeed, this means that the only
singularities in codimension 1 are nodes,
which can only deform to smooth points.  We therefore
assume that the fibers of $f$ have normal crossings
in codimension one.

\subsection{Constructibility}
\begin{lemma}
Under the assumptions of Proposition~\ref{prop:SLC},
the locus $B_{\slc}$ is constructible in $B$.
\end{lemma}
\begin{proof}
The main ingredient is resolution of singularities.
The image of a constructible set is constructible,
so we may replace $B$ by a resolution of singularities
of its normalization $B' \rightarrow  B$.  We have the pull
back family
$$X'=X \times_B B' \rightarrow  B'.$$
We claim we can stratify $B'$ by locally-closed
subsets 
$$B'=\coprod_{j\in J} B'_j$$
such that the restriction over each stratum
$$f_j:X'_j=X'|B'_j \rightarrow B'_j$$
admits a simultaneous good resolution of singularities.  
By definition, a good resolution of singularities is one where the preimage
of the singular locus is a simple normal crossings divisor;
in a simultaneous good resolution, each intersection of components
of the simple normal crossings divisor is smooth over the base.

The stratification is constructed inductively:  Choose a good
resolution of singularities $W \rightarrow X'$, which induces a 
good resolution of the geometric generic fiber of $X' \rightarrow B'$.
Since we are working over a field of characteristic zero, such fiber is smooth.  We can choose a dense open subset $B''\subset B'$ such that
$W\rightarrow X'$ is a simultaneous good resolution of the fibers $X'_b$
for $b\in B''$.  However, $B'\setminus B''\subsetneq B'$ is closed, so we are done by Noetherian induction.  

The fibers of $f_j$ have invertible dualizing sheaves, so 
Lemma~\ref{lemma:SLCcrit} allows us to determine whether
the fibers have SLC singularities by computing discrepancies
on the good resolutions.  However, discrepancies are constant
over families with a simultaneous good resolution, so if one fiber
of $f_j$ is SLC then all fibers are SLC.  In particular, the 
the fibers are SLC over a constructible subset of the base.  
\end{proof}

\subsection{Proof of Proposition}
We complete the proof that $B_{\slc}$ is open by proving that it is 
stable under one-parameter generalizations.  Precisely, let $T$ be 
nonsingular connected curve, $t\in T$ a closed point,
$\mu:T\rightarrow  B$ a morphism, and
$$f_T:X_T=X\times_B T \rightarrow T$$
the base change of our family to $T$.  Then the set
$$\{s\in T: X_s \text{ is semilog canonical } \}$$
is open.  The existence of the sheaf $\cL$ 
guarantees that the canonical class of $X_T$ is $\QQ$-Cartier.

Suppose that $s\in T$ is such that $X_s$ has semilog canonical singularities.
Let $\nu:\tilde{X}_T \to X_T$ denote the normalization and $g:\tilde{X}\to T$ the 
induced morphism.  Let $C$ denote the conductor divisor, which
is also flat over $T$.  We may write
$$\nu^*K_X=K_{\tilde{X}}+C,$$
which is also $\QQ$-Cartier.  
If $S$ is the normalization of some irreducible
component of $\tilde{X}_s$ and $\iota:S \to \tilde{X}_s$ the
induced morphism, then 
$$\iota^*(K_{\tilde{X}_s}+C_s)=K_S+\Theta,$$
where $\Theta$ contains the conductor and the preimage of $C_s$.
Lemma~\ref{lemma:SLCcrit} implies that $(S,\Theta)$ is 
log canonical, and therefore that $(\tilde{X}_s,C_s)$ is semilog canonical.
By `Inversion of Adjunction' \cite{kawakita}, the pair 
$(\tilde{X},\tilde{X}_s+C)$ is log canonical as well, at least
in a neighborhood of $\tilde{X}_s$.  Fix a good log resolution
$$\rho:(Y,D) \to (\tilde{X},\tilde{X}_s+C),$$
i.e., one where the union of the exceptional locus and the proper
transform of the boundary is simple normal crossings.  
All discrepancies are computed with respect to this resolution.  
Since $\tilde{X}_s$ is a fiber of $\tilde{X}\to T$, we have
\begin{itemize}
\item
the log discrepancies of $(\tilde{X},C)$ for exceptional divisors
with center in $\tilde{X}_s$ are $\ge 0$;
\item
the log discrepancies of $(\tilde{X},C)$ for exceptional divisors
dominating $T$ are $\ge -1$;
\end{itemize}

For $t\neq s$ in a neighborhood of $s$, the fiber $\tilde{X}_t$ is normal.
Furthermore, $\rho$ induces a good log resolution of $(\tilde{X}_t,C_t)$.  
Consider the pairs $(\tilde{X},\tilde{X}_t+C)$, which are also log
canonical by the discrepancy analysis above.  Applying `Adjunction', we 
conclude that $(\tilde{X}_t,C_t)$ is also log canonical.  
Lemma~\ref{lemma:SLCcrit} implies that $X_t$ is therefore semilog canonical.
\qed

\bibliographystyle{alpha}

\bibliography{slcmod}

\begin{thebibliography}{KKMSD73}

\bibitem[Abr02]{A-pluri}
Dan Abramovich.
\newblock Canonical models and stable reduction for plurifibered varieties.
\newblock Unpublished, {\tt arXiv:math/0207004}, 2002.

\bibitem[ACV03]{ACV}
Dan Abramovich, Alessio Corti, and Angelo Vistoli.
\newblock Twisted bundles and admissible covers.
\newblock {\em Comm. Algebra}, 31(8):3547--3618, 2003.
\newblock Special issue in honor of Steven L. Kleiman.

\bibitem[AGV08]{AGV05}
Dan Abramovich, Tom Graber, and Angelo Vistoli.
\newblock Gromov-{W}itten theory of {D}eligne-{M}umford stacks.
\newblock {\em Amer. J. Math.}, 130(5):1337--1398, 2008.

\bibitem[AJ03]{AJ}
Dan Abramovich and Tyler~J. Jarvis.
\newblock Moduli of twisted spin curves.
\newblock {\em Proc. Amer. Math. Soc.}, 131(3):685--699 (electronic), 2003.

\bibitem[Ale94]{Alexeev}
Valery Alexeev.
\newblock Boundedness and {$K\sp 2$} for log surfaces.
\newblock {\em Internat. J. Math.}, 5(6):779--810, 1994.

\bibitem[AOV08a]{AOV}
Dan Abramovich, Martin Olsson, and Angelo Vistoli.
\newblock Tame stacks in positive characteristic.
\newblock {\em Ann. Inst. Fourier (Grenoble)}, 58(4):1057--1091, 2008.

\bibitem[AOV08b]{AOVtwisted}
Dan Abramovich, Martin Olsson, and Angelo Vistoli.
\newblock Twisted stable maps to tame {A}rtin stacks.
\newblock preprint {\tt arXiv:0801.3040v2}, 2008.

\bibitem[Art74]{Artin}
M.~Artin.
\newblock Versal deformations and algebraic stacks.
\newblock {\em Invent. Math.}, 27:165--189, 1974.

\bibitem[AV00]{AV1}
Dan Abramovich and Angelo Vistoli.
\newblock Complete moduli for fibered surfaces.
\newblock In {\em Recent progress in intersection theory (Bologna, 1997)},
  Trends Math., pages 1--31. Birkh\"auser Boston, Boston, MA, 2000.

\bibitem[AV02]{AV2}
Dan Abramovich and Angelo Vistoli.
\newblock Compactifying the space of stable maps.
\newblock {\em J. Amer. Math. Soc.}, 15(1):27--75 (electronic), 2002.

\bibitem[BCHM06]{BCHM}
Caucher Birkar, Paolo Cascini, Christopher~D. Hacon, and James McKernan.
\newblock Existence of minimal models for varieties of log general type.
\newblock preprint {\tt arXiv:math/0610203}, 2006.

\bibitem[Con00]{Conrad}
Brian Conrad.
\newblock {\em Grothendieck duality and base change}, volume 1750 of {\em
  Lecture Notes in Mathematics}.
\newblock Springer-Verlag, Berlin, 2000.

\bibitem[Elk81]{Elkik}
Ren{\'e}e Elkik.
\newblock Rationalit\'e des singularit\'es canoniques.
\newblock {\em Invent. Math.}, 64(1):1--6, 1981.

\bibitem[Gro66]{EGA-IV-3}
A.~Grothendieck.
\newblock \'{E}l\'ements de g\'eom\'etrie alg\'ebrique. {IV}. \'{E}tude locale
  des sch\'emas et des morphismes de sch\'emas. {III}.
\newblock {\em Inst. Hautes \'Etudes Sci. Publ. Math.}, (28):255, 1966.

\bibitem[Hac04]{Hacking}
Paul Hacking.
\newblock Compact moduli of plane curves.
\newblock {\em Duke Math. J.}, 124(2):213--257, 2004.

\bibitem[HK04]{HK}
Brendan Hassett and S{\'a}ndor~J. Kov{\'a}cs.
\newblock Reflexive pull-backs and base extension.
\newblock {\em J. Algebraic Geom.}, 13(2):233--247, 2004.

\bibitem[IR96]{IR}
Yukari Ito and Miles Reid.
\newblock The {M}c{K}ay correspondence for finite subgroups of {${\rm
  SL}(3,\mathbf C)$}.
\newblock In {\em Higher-dimensional complex varieties ({T}rento, 1994)}, pages
  221--240. de Gruyter, Berlin, 1996.

\bibitem[Jar00]{Jarvis}
Tyler~J. Jarvis.
\newblock Geometry of the moduli of higher spin curves.
\newblock {\em Internat. J. Math.}, 11(5):637--663, 2000.

\bibitem[Kar00]{Karu}
Kalle Karu.
\newblock Minimal models and boundedness of stable varieties.
\newblock {\em J. Algebraic Geom.}, 9(1):93--109, 2000.

\bibitem[Kaw99a]{KawCan}
Yujiro Kawamata.
\newblock Deformations of canonical singularities.
\newblock {\em J. Amer. Math. Soc.}, 12(1):85--92, 1999.

\bibitem[Kaw99b]{KawamataPathology}
Yujiro Kawamata.
\newblock Index 1 covers of log terminal surface singularities, 1999.

\bibitem[Kaw07]{kawakita}
Masayuki Kawakita.
\newblock Inversion of adjunction on log canonicity.
\newblock {\em Invent. Math.}, 167(1):129--133, 2007.

\bibitem[KK09]{KolKov}
J{\'a}nos Koll{\'a}r and Sandor~J. Kov{\'a}cs.
\newblock Log canonical singularities are {D}u {B}ois, 2009.
\newblock preprint {\tt arXiv:0902.0648}.

\bibitem[KKMSD73]{KKMS}
G.~Kempf, Finn~Faye Knudsen, D.~Mumford, and B.~Saint-Donat.
\newblock {\em Toroidal embeddings. {I}}.
\newblock Lecture Notes in Mathematics, Vol. 339. Springer-Verlag, Berlin,
  1973.

\bibitem[KM97]{KeMo}
Se{\'a}n Keel and Shigefumi Mori.
\newblock Quotients by groupoids.
\newblock {\em Ann. of Math. (2)}, 145(1):193--213, 1997.

\bibitem[KM98]{KoMo}
J{\'a}nos Koll{\'a}r and Shigefumi Mori.
\newblock {\em Birational geometry of algebraic varieties}, volume 134 of {\em
  Cambridge Tracts in Mathematics}.
\newblock Cambridge University Press, Cambridge, 1998.
\newblock With the collaboration of C. H. Clemens and A. Corti, Translated from
  the 1998 Japanese original.

\bibitem[Kol86]{KollarHDI}
J{\'a}nos Koll{\'a}r.
\newblock Higher direct images of dualizing sheaves. {I}.
\newblock {\em Ann. of Math. (2)}, 123(1):11--42, 1986.

\bibitem[Kol90]{KollarJDG}
J{\'a}nos Koll{\'a}r.
\newblock Projectivity of complete moduli.
\newblock {\em J. Differential Geom.}, 32(1):235--268, 1990.

\bibitem[Kol92]{FlAb}
J\'anos Koll\'ar, editor.
\newblock {\em Flips and abundance for algebraic threefolds}.
\newblock Soci\'et\'e Math\'ematique de France, Paris, 1992.
\newblock Papers from the Second Summer Seminar on Algebraic Geometry held at
  the University of Utah, Salt Lake City, Utah, August 1991, Ast{\'e}risque No.
  211 (1992).

\bibitem[Kol95]{kollarjalg}
J{\'a}nos Koll{\'a}r.
\newblock Flatness criteria.
\newblock {\em J. Algebra}, 175(2):715--727, 1995.

\bibitem[Kol08]{KollarHH}
J\'anos Koll\'ar.
\newblock Hulls and husks.
\newblock preprint {\tt arXiv:0801.3040v2}, 2008.

\bibitem[KSB88]{K-SB}
J.~Koll{\'a}r and N.~I. Shepherd-Barron.
\newblock Threefolds and deformations of surface singularities.
\newblock {\em Invent. Math.}, 91(2):299--338, 1988.

\bibitem[LMB00]{LM}
G{\'e}rard Laumon and Laurent Moret-Bailly.
\newblock {\em Champs alg\'ebriques}, volume~39 of {\em Ergebnisse der
  Mathematik und ihrer Grenzgebiete. 3. Folge. A Series of Modern Surveys in
  Mathematics [Results in Mathematics and Related Areas. 3rd Series. A Series
  of Modern Surveys in Mathematics]}.
\newblock Springer-Verlag, Berlin, 2000.

\bibitem[Nak86]{nakayama}
Noboru Nakayama.
\newblock Invariance of the plurigenera of algebraic varieties under minimal
  model conjectures.
\newblock {\em Topology}, 25(2):237--251, 1986.

\bibitem[Nir09]{Nironi}
Fabio Nironi, 2009.
\newblock SISSA Ph.D. thesis, in preparation.

\bibitem[OS03]{OS}
Martin Olsson and Jason Starr.
\newblock Quot functors for {D}eligne-{M}umford stacks.
\newblock {\em Comm. Algebra}, 31(8):4069--4096, 2003.
\newblock Special issue in honor of Steven L. Kleiman.

\bibitem[Rei78]{Reid78}
Miles Reid.
\newblock Surfaces with {$p\sb{g}=0$}, {$K\sp{2}=1$}.
\newblock {\em J. Fac. Sci. Univ. Tokyo Sect. IA Math.}, 25(1):75--92, 1978.

\bibitem[Rei80]{Re80Ca3folds}
Miles Reid.
\newblock Canonical {$3$}-folds.
\newblock In {\em Journ\'ees de G\'eometrie Alg\'ebrique d'Angers, Juillet
  1979/Algebraic Geometry, Angers, 1979}, pages 273--310. Sijthoff \&
  Noordhoff, Alphen aan den Rijn, 1980.

\bibitem[Rei87]{ReYPG}
Miles Reid.
\newblock Young person's guide to canonical singularities.
\newblock In {\em Algebraic geometry, {B}owdoin, 1985 ({B}runswick, {M}aine,
  1985)}, volume~46 of {\em Proc. Sympos. Pure Math.}, pages 345--414. Amer.
  Math. Soc., Providence, RI, 1987.

\bibitem[Rei97]{Reid-Chapters}
Miles Reid.
\newblock Chapters on algebraic surfaces.
\newblock In {\em Complex algebraic geometry ({P}ark {C}ity, {UT}, 1993)},
  volume~3 of {\em IAS/Park City Math. Ser.}, pages 3--159. Amer. Math. Soc.,
  Providence, RI, 1997.

\bibitem[Rei02]{Reid-Graded}
Miles Reid.
\newblock Graded rings and varieties in weighted projective space, 2002.
\newblock unpublished chapter from an upcoming book on surfaces, available at
  {\tt http://www.warwick.ac.uk/\!$\sim$masda/surf/more/grad.pdf}.

\bibitem[Rom05]{Romagny}
Matthieu Romagny.
\newblock Group actions on stacks and applications.
\newblock {\em Michigan Math. J.}, 53(1):209--236, 2005.

\bibitem[RT09]{Ross-Thomas}
Julius Ross and Richard Thomas.
\newblock Weighted projective embeddings, stability of orbifolds and constant
  scalar curvature {K}\"ahler metrics, 2009.
\newblock manuscript in preparation.

\bibitem[Sin03]{Singh}
Anurag~K. Singh.
\newblock Cyclic covers of rings with rational singularities.
\newblock {\em Trans. Amer. Math. Soc.}, 355(3):1009--1024 (electronic), 2003.

\bibitem[Vie95]{Viehweg}
Eckart Viehweg.
\newblock {\em Quasi-projective moduli for polarized manifolds}, volume~30 of
  {\em Ergebnisse der Mathematik und ihrer Grenzgebiete (3) [Results in
  Mathematics and Related Areas (3)]}.
\newblock Springer-Verlag, Berlin, 1995.

\end{thebibliography}

\end{document}